\documentclass[12pt, a4paper,reqno]{amsart}

\usepackage{amsthm, comment}
{\theoremstyle{definition}
\newtheorem{dfn}{Definition}[section]}
\newtheorem{prop}[dfn]{Proposition}

\newtheorem{thm}[dfn]{Theorem}
{\theoremstyle{definition}
\newtheorem{rem}[dfn]{Remark}}
\newtheorem{lem}[dfn]{Lemma}
\newtheorem{cor}[dfn]{Corollary}

\newtheorem{ques}[dfn]{Question}
{\theoremstyle{definition}
\newtheorem{exa}[dfn]{Example}}

\usepackage[top=30truemm,bottom=30truemm,left=35truemm,right=35truemm]{geometry}
\usepackage{amsmath,amssymb,amscd,bm,graphicx,tikz-cd,upgreek,dsfont,amsfonts,tikz,enumitem,euscript, stmaryrd,titletoc}
\usetikzlibrary{cd}
\usepackage[all]{xy}
\usepackage[colorlinks=true,
    citecolor=green(pigment),
    linkcolor=alizarin,
    urlcolor=denim]{hyperref}

\usepackage{hyperref}
    
\definecolor{alizarin}{rgb}{0.82, 0.1, 0.26}
\definecolor{azure(colorwheel)}{rgb}{0.0, 0.5, 1.0}
\definecolor{blue(pigment)}{rgb}{0.2, 0.2, 0.6}
\definecolor{denim}{rgb}{0.08, 0.38, 0.74}
\definecolor{mint}{rgb}{0.24, 0.71, 0.54}
\definecolor{parisgreen}{rgb}{0.31, 0.78, 0.47}
\definecolor{persiangreen}{rgb}{0.0, 0.65, 0.58}
\definecolor{seagreen}{rgb}{0.18, 0.55, 0.34}
\definecolor{shamrockgreen}{rgb}{0.0, 0.62, 0.38}
\definecolor{green(pigment)}{rgb}{0.0, 0.65, 0.31}

\usepackage[normalem]{ulem}

\usepackage{setspace}
\newcommand{\marginparstretch}{0.6}
\let\oldmarginpar\marginpar
\renewcommand\marginpar[1]{\-\oldmarginpar[\framebox{\setstretch{\marginparstretch}\begin{minipage}{\marginparwidth}{\raggedleft\tiny #1}\end{minipage}}]{\framebox{\setstretch{\marginparstretch}\begin{minipage}{\marginparwidth}{\raggedright\tiny #1}\end{minipage}}}}

\usepackage{scalerel,stackengine}
\stackMath
\newcommand\reallywidecheck[1]{%
\savestack{\tmpbox}{\stretchto{%
   \scaleto{%
    \scalerel*[\widthof{\ensuremath{#1}}]{\kern-.6pt\bigwedge\kern-.6pt}%
    {\rule[-\textheight/2]{1ex}{\textheight}}
   }{\textheight}%
}{0.5ex}}%
\stackon[1pt]{#1}{\scalebox{-1}{\tmpbox}}%
}
\def\ang<#1>{\langle #1 \rangle}
\def\bigang<#1>{\left\langle #1 \right\rangle}

\numberwithin{equation}{section}

\newcommand{\Dsg}{\operatorname{D^{sg}}}
\newcommand{\Sing}{\operatorname{Sing}}
\newcommand{\Spec}{\operatorname{Spec}}
\newcommand{\Supp}{\operatorname{Supp}}
\newcommand{\Hom}{\operatorname{Hom}}
\newcommand{\Ext}{\operatorname{Ext}}

\newcommand{\REnd}{\operatorname{{\mathbf R}End}}
\newcommand{\End}{\operatorname{End}}

\newcommand{\Max}{\operatorname{Max}}

\newcommand{\fmod}{\operatorname{mod}}

\newcommand{\vect}{\operatorname{vect}}

\newcommand{\Db}{\operatorname{D^b}}
\newcommand{\D}{\operatorname{D}}

\newcommand{\Perf}{\operatorname{Perf}}

\newcommand{\id}{\operatorname{id}}

\newcommand{\CS}{\operatorname{CS}}
\newcommand{\Spcl}{\operatorname{Spcl}}

\newcommand{\Hilb}{\operatorname{Hilb}}
\newcommand{\Sym}{\operatorname{Sym}}

\newcommand{\Ker}{\operatorname{Ker}}

\newcommand{\Th}{\mathbf{Th}}
\renewcommand{\Max}{\operatorname{Max}}

\newcommand{\rk}{\operatorname{\sf rk}}

\newcommand{\pr}{\operatorname{\sf pr}}

\newcommand{\brick}{\operatorname{\sf Bricks}}
\newcommand{\wide}{\operatorname{\sf wide}}

\newcommand{\LS}{\operatorname{LS}}
\newcommand{\JD}{\operatorname{JD}}

\def\Kr{\mathop{\sf Kr}\nolimits}

\newcommand{\sfT}{{\sf T}}
\newcommand{\sfA}{{\sf A}}

\newcommand{\cC}{\mathcal{C}}

\newcommand{\cE}{\mathcal{E}}

\newcommand{\cO}{\mathcal{O}}

\newcommand{\cS}{\mathcal{S}}
\newcommand{\cT}{\mathcal{T}}

\newcommand{\scrA}{\EuScript{A}}
\newcommand{\scrB}{\EuScript{B}}
\newcommand{\scrC}{\EuScript{C}}

\newcommand{\scrI}{\EuScript{I}}

\newcommand{\scrL}{\EuScript{L}}
\newcommand{\scrM}{\EuScript{M}}
\newcommand{\scrN}{\EuScript{N}}

\newcommand{\scrP}{\EuScript{P}}
\newcommand{\scrQ}{\EuScript{Q}}
\newcommand{\scrR}{\EuScript{R}}
\newcommand{\scrS}{\EuScript{S}}
\newcommand{\scrT}{\EuScript{T}}
\newcommand{\scrU}{\EuScript{U}}
\newcommand{\scrV}{\EuScript{V}}
\newcommand{\scrW}{\EuScript{W}}

\newcommand{\scrX}{\EuScript{X}}

\newcommand{\scrKr}{\EuScript{K}r}

\newcommand{\bA}{\mathbb{A}}
\newcommand{\bC}{\mathbb{C}}
\newcommand{\bE}{\mathbb{E}}
\newcommand{\bF}{\mathbb{F}}

\newcommand{\bN}{\mathbb{N}}
\newcommand{\bP}{\mathbb{P}}

\newcommand{\bS}{\mathbb{S}}

\newcommand{\bZ}{\mathbb{Z}}

\newcommand{\m}{\mathfrak{m}}

\newcommand{\p}{\mathfrak{p}}

\newcommand{\simto}{\xrightarrow{\sim}}
\newcommand{\hookto}{\hookrightarrow}
\newcommand{\surjto}{\twoheadrightarrow}

\renewcommand{\l}{\langle}
\renewcommand{\r}{\rangle}

\newcommand{\spec}{\mathrm{ Spec}_{\triangle}}

\newcommand{\ellu}{\ell_{\rm ult}}

\newcommand{\lbr}{\llbracket}
\newcommand{\rbr}{\rrbracket}

\newcommand{\relmiddle}[1]{\mathrel{}\middle#1\mathrel{}}

\tikzset{
        DB/.style={circle,draw=black,circle,fill=white,inner sep=0pt, minimum size=4pt},
        DW/.style={circle,draw=black,fill=black,inner sep=0pt, minimum size=4pt},
        cvertex/.style={circle,draw=black,fill=white,inner sep=1pt,outer sep=3pt},
        vertex/.style={circle,fill=black,inner sep=1pt,outer sep=3pt},
        star/.style={circle,fill=yellow,inner sep=0.75pt,outer sep=0.75pt},
        tvertex/.style={inner sep=1pt,font=\scriptsize},
	pvertex/.style={circle,inner sep=1pt,outer sep=2pt,font=\scriptsize},
        gap/.style={inner sep=0.5pt,fill=white}
}

\makeatletter
\newcommand*{\defeq}{\mathrel{\rlap{%
                     \raisebox{0.3ex}{$\m@th\cdot$}}%
                     \raisebox{-0.3ex}{$\m@th\cdot$}}%
                     =}
\makeatother

\makeatletter
\let\@wraptoccontribs\wraptoccontribs
\makeatother

\usepackage{xr}

\begin{document}
\title[]{Length of triangulated categories}

\author[Y.~Hirano]{Yuki Hirano}
\author[M.~Kalck]{Martin Kalck}
\author[G.~Ouchi]{Genki Ouchi}

\address{Y.~Hirano, Tokyo University of Agriculture and Technology, 2--24--16 Nakacho, Koganei, Tokyo 184--8588, Japan}
\email{hirano@go.tuat.ac.jp}

\address{M.~Kalck, Institut f\"ur Mathematik und Wissenschaftliches Rechnen, 
Universit\"at Graz,
Heinrichstrasse 36, 8010 Graz, Austria}
\email{martin.kalck@uni-graz.at}

\address{G.~Ouchi, Department of Mathematics, Faculty of Science, Hokkaido University, Kita 10, Nishi 8, Kita-Ku, Sapporo, Hokkaido, 060-0810, Japan}
\email{genki.ouchi@math.sci.hokudai.ac.jp}

\begin{abstract}
We introduce the notion of composition series of triangulated categories, which generalizes full exceptional sequences. The \emph{lengths} of composition series yield invariants for triangulated categories.

We  study composition series of derived categories for some classes of projective varieties and finite-dimensional algebras. 
{We prove that certain negative rational curves on rational surfaces cause composition series of different lengths in the derived categories of the surfaces.}
On the other hand, we show that
for derived categories of finite-dimensional hereditary algebras, for nontrivial admissible subcategories of $\Db(\bP^2)$ and for derived categories of some singular varieties, all composition series have the same length.
\end{abstract}

\maketitle{}



\tableofcontents

\section{Introduction}

\subsection{Background and motivation}~\label{Section1.1}

Triangulated categories are actively studied in areas including algebraic geometry, algebraic topology, symplectic geometry, representation theory and mathematical physics. In general, they are far too complex to be able to understand and classify all their objects in detail. This motivates the study of coarser structures like thick subcategories.

The first celebrated result in this area, was the classification of thick subcategories of compact objects in  $p$-local stable homotopy categories by Devinatz, Hopkins and Smith \cite{dhs,hs} in the 1980s. Inspired by this, in an algebro-geometric context, Hopkins \cite{hop}  and Neeman \cite{neeman} classified all thick subcategories, which are automatically $\otimes$-ideals, of perfect complexes over noetherian rings. This result was extended to the classification of $\otimes$-ideals of perfect complexes over  quasi-compact and quasi-separated schemes by Thomason \cite{thomason}. In turn, Thomason's result is the starting point for Balmer's tensor triangular geometry \cite{balmer}, extracting geometric information out of monoidal triangulated categories. 

A classification of thick subcategories is also known for stable categories of maximal Cohen--Macaulay modules over hypersurface singularities (equivalently homotopy categories of matrix factorizations) by Takahashi \cite{takahashi1}, in which case thick subcategories are automatically $\otimes$-submodules. This was extended to $\otimes$-submodules of singularity categories of complete intersections by Stevenson \cite{stevenson} and to derived matrix factorization categories by the first named author \cite{hirano}. 
Moreover, the lattice of thick subcategories of $\Db(X)$ contains information about Fourier-Mukai partners of $X$ and autoequivalences of $\Db(X)$, cf.\ \cite{ho1,ho2,ito,im}, building on work of Matsui \cite{matsui,matsui2}.

However, despite of many efforts on the study of thick subcategories, lattices of thick subcategories for derived categories of projective varieties are still quite poorly understood. In fact, the lattice of thick subcategories and the Matsui spectrum are determined only for the projective line and for elliptic curves, and admissible subcategories, which are very special thick subcategories, on the projective plane have only been classified very recently \cite{pir}. 
Therefore, instead of studying the entire lattice of thick subcategories, we focus on  much coarser invariants of this lattice. These invariants yield new invariants for  triangulated categories and will be described in more detail in the following subsections.

\vspace{2mm}
\subsection{Composition series of triangulated categories}~

Composition series are fundamental in the study of finite groups and modules. One of the most important properties of composition series of finite groups or finite length modules is the Jordan--H\"older (JH) property, implying, in particular, that the length of different composition series coincide (we call this the \emph{Jordan--Dedekind (JD) property} below).  

In this paper, we introduce the notion of composition series for triangulated categories as maximal chains in the lattice of thick subcategories. By definition, every full exceptional sequence gives rise to a composition series whose length is the rank of the Grothendieck group. 
We study the lengths of composition series in many examples from algebraic geometry and representation theory. 
In particular, we show that the JD property holds for all path algebras of acyclic quivers 
and thus for
(certain orbifold) projective lines.

In general, the lattice of thick subcategories has a much richer structure, and we observe that there can be composition series of different lengths. We collect all possible lengths of composition series of a fixed category $\scrT$ into a set, denoted by $\LS(\scrT)$, that we call the \emph{length spectrum} of $\scrT$. {Then the {\it length} of $\scrT$ is defined to be the minimum integer in $\LS(\scrT)$.} 

We  illustrate this by showing that the length spectra of derived categories of certain rational surfaces and certain threefolds are not singletons. A key common feature of many of these examples is that they are small resolutions of singular projective varieties, and their derived categories contain admissible subcategories $\scrA$ with composition series of lengths two and three ($\scrA$ depends on the varieties). More precisely, $\scrA$ is generated by two exceptional objects (yielding a composition series of length two) and contains a bouquet sphere-like object (see Section \ref{section:1.2} for the definition) that contributes to a composition series of length three.
Moreover, $\scrA$ can be viewed as a categorical resolution of singularities, 
cf.\ e.g.\ \cite{kks, ks}. For the Hirzebruch surface $\bF_2$ and certain small resolutions of nodal threefolds, these categories $\scrA$ 
belong to a well-known family of triangulated categories of ``discrete representation type", cf.\ e.g.\ \cite{bgs, ky18}.  
Furthermore,  we show that the derived category of a smooth toric surface with a ($-m$)-curve for some $m>1$ does not satisfy the JD property
 -- again a bouquet sphere-like object plays a key role.

\vspace{2mm}
\subsection{Main results}~\label{section:1.2}

Let $\scrT$ be an essentially small triangulated category.  
If $\ell(\scrT)<\infty$, one of the basic questions is: does $\scrT$ satisfy the JD property?   In a draft version of this paper, we conjectured that the derived category of coherent sheaves on a smooth projective rational surface satisfies the JD property. 
We show that this conjecture does not hold in general. More precisely, we prove the following.

\begin{thm}[Corollary \ref{cor:toric surface}]\label{thm:intro} Let $X$ be a smooth projective toric surface containing a rational curve $C$ with $C^2< -1$. Then  the derived category $\Db(X)$ contains an admissible subcategory $\scrA$  such that $\ell(\scrA)<\infty$,  $\ell(\Db(X)/\scrA)<\infty$ and $\scrA$ does not have  the JD property. In particular, $\Db(X)$  does not satisfy the JD property. 
\end{thm}

Moreover, we can extract  information on negative rational curves from composition series: 

\begin{thm}[Theorem \ref{thm:toric}]
Let $X$ be a smooth projective toric surface, and put $r\defeq \rk K_0(\Db(X))$.  
If there are disjoint rational curves $C_1,\cdots,C_n$ with $C_i^2<-1$ for all $1\leq i\leq n$, then 
\[
\{r,r+1,\hdots, r+n\}\subseteq \LS(\Db(X)).
\]
\end{thm}

The key observation for the above results is that $\Db(X)$ admits an admissible subcategory $\widetilde{\scrA}$ containing a {bouquet sphere-like} object $\cO_C(-1)$ that induces a composition series whose length is greater than $\rk(K_0(\widetilde{\scrA}))$. Here an {\it $n$-bouquet $d$-sphere-like object} is an object $B$ whose graded endomorphism algebra $\Hom^*(B,B)$ is $n+1$-dimensional, with one-dimensional degree $0$ part and $n$-dimensional degree $d$ part.
The object $\cO_{C}(-1)$ is bouquet 2-sphere-like if and only if $C^2< -1$. 
On the other hand,  in contrast to the surfaces in Theorem \ref{thm:intro}, the blow-up $X$ of $\bP^2_{\bC}$ at a finite set of points in {\it very general position} (see \cite[Definition 2.1]{fer} for the definition)  does not contain any rational curve $C$ with $C^2< -1$ \cite[Proposition 2.3]{fer}, and $\Db(X)$ does not admit any spherical object\footnote{We do not know whether these categories admit bouquet 2-sphere-like objects.} \cite[Theorem 1.1]{hk}. This observation leads us to the following:

\begin{ques}\label{ques:krah}
Let $X$ be the blow-up of $\bP^2_{\bC}$ at a finite set of points in very general position. Does $\Db(X)$ satisfy the JD property?
\end{ques}

If the answer to this question is affirmative,  phantom subcategories constructed by Krah \cite{kra} are of  infinite length.

We provide further counterexamples to the JD property, namely the derived category $\Db(\Lambda)$ of a certain finite-dimensional algebra $\Lambda$ of finite global dimension, which is  {\it  derived-discrete}, cf.\ \cite{bpp} and Section \ref{subs: derived-discrete}.

\begin{thm}[Corollary \ref{cor:appendix}]\label{thm:intro2}
Let $\Lambda$ be a connected finite-dimensional $k$-algebra of finite global dimension, and assume that $\Lambda$ is derived-discrete. Then  $\Db(\Lambda)$ satisfies the JD property if and only if 
 $\Lambda$ is derived equivalent to  the path algebra  of a Dynkin quiver.
\end{thm}

On the other hand, we prove that certain triangulated categories satisfy the JD property -- namely, the derived categories $\Db(kQ)$ of acyclic quivers $Q$. More precisely, we show the following statement using a recent result
by Asai \cite{asai}. Interestingly, in this case, we are able to classify all composition series without using a classification of all thick subcategories.

\begin{thm}[Theorem \ref{thm:JD wild}] \label{Thm:1.5.Intro}
{Let $Q$ be a finite acyclic quiver.
Then every composition series of $\Db(kQ)$ arises from a full exceptional sequence. In particular, $\Db(kQ)$ satisfies the JD property.}
\end{thm}

In combination with \cite{pir,bp2}, this implies the following result.

\begin{cor}\label{Cor:1.6.Intro}
Admissible subcategories $\scrA \subsetneq \Db(\bP^2)$ satisfy the JD-property.
\end{cor}

This shows that the lattices of thick subcategories of $\bP^2$ and a toric surface with a negative curve $C$ with $C^2<-1$ have different properties. This might be an evidence of the affirmative answer to Question \ref{ques:krah} since the surfaces in the question never have any rational curve $C$ with $C^2<-1$.

\vspace{2mm}

\vspace{1mm}
\subsection{Notations and conventions}~

Let $k$ be a field. All categories are $k$-linear and essentially small. Functors between triangulated categories are  $k$-linear and exact.
We denote the Verdier quotient of a triangulated category $\scrT$ by a thick subcategory $\scrU$ by $\scrT/\scrU$.
Subcategories are assumed to be closed under isomorphisms.
$\Db(X)$ denotes the bounded derived category of coherent sheaves on a scheme $X$, and $\Perf X\subseteq \Db(X)$ denotes the thick subcategory of perfect complexes on $X$. 
For a (not necessarily commutative) noetherian ring $\Lambda$, we denote the bounded derived category of finitely generated right $\Lambda$-modules by $\Db(\Lambda)$, and write $\Perf \Lambda$ for the full subcategory of perfect complexes over $\Lambda$. For finitely many elements $a_1,\hdots,a_n\in\Lambda$, we denote by $\l a_1,\hdots, a_n\r$ the two-sided ideal of $\Lambda$ generated by $a_1,\hdots,a_n$.
For a proper morphism $f\colon X\to Y$ of smooth varieties, we write $f_*\colon \Db(X)\to \Db(Y)$ and  $f^*\colon \Db(Y)\to \Db(X)$ for the derived push-forward and the derived pull-back, respectively. Similarly, we denote by $\otimes_X$ the derived tensor product on $X$.
Points on varieties (or schemes) are not necessarily closed.

\vspace{3mm}
\subsection{Acknowledgements}~

We would like to thank Greg Stevenson for informing us about the existence of triangulated categories without the JD property. We also thank Hiroki Matsui, Nebojsa Pavic, Alexey Elagin, Souvik Dey, Nathan Broomhead and Sota Asai for giving valuable comments. 
We are very grateful to the referee for reading our text very carefully. Their many comments and suggestions greatly improved the quality of this work.
Y.H. is supported by JSPS KAKENHI Grant Number 23K12956. G.O. is supported by JSPS KAKENHI Grant Number 19K14520.

\vspace{3mm}
\section{Length of triangulated categories}\label{section:length}

 We assume readers are familiar with notions of full exceptional sequences, admissible subcategories and semi-orthogonal decompositions of triangulated categories. See e.g.\ \cite{huy} for these notions. Throughout this section, $\scrT$ denotes a triangulated category.

\vspace{2mm}
\subsection{Semi-simple triangulated categories} 
~

 An object $A\in \scrT$ is called a {\it direct summand} of $B\in\scrT$ if there is an object $A'\in \scrT$ and an isomorphism $A\oplus A'\cong B$. A {\it thick} subcategory of $\scrT$ is a full triangulated subcategory that is closed under taking direct summands.
The set of thick subcategories of $\scrT$, denoted by $\Th(\scrT)$,  is partially ordered by inclusions. Consider a subset $\{\scrU_i\}_{i\in I}\subset \Th(\scrT)$. Then  the intersection $\bigwedge_{i\in I}\scrU_i\defeq\bigcap_{i\in I}\scrU_i$ is the infimum of $\{\scrU_i\}_{i\in I}$.  Dually,  $\bigvee_{i\in I}\scrU_i\defeq \bigwedge_{\scrU_i\subseteq \scrW}\scrW$, which is the intersection of all thick subcategories $\scrW\in \Th(\scrT)$ containing all $\scrU_i$, is the supremum of $\{\scrU_i\}_{i\in I}$. Thus the poset $\Th(\scrT)$ is a complete lattice. 

For any  collection $\cC$ of objects in $\scrT$, we denote by 
\[
\lbr \cC\rbr\in\Th(\scrT)
\]
the thick subcategory generated by $\cC$, and we write $[\cC]$  for the triangulated subcategory generated by $\cC$. As usual, 
for $\cC_1,\hdots,\cC_n\subset\scrT$ and $A_1,\hdots,A_n\in \scrT$, we write $\lbr\cC_1,\hdots,\cC_n\rbr\defeq \lbr\cC_1\cup\cdots\cup\cC_n\rbr$ and $\lbr A_1,\hdots,A_n\rbr\defeq\lbr\{A_1,\hdots,A_n\}\rbr$. The same notation is used for $[-]$.

\begin{dfn}An object $G\in\scrT$ is called a {\it split generator} of $\scrT$ if $\scrT=\lbr G\rbr$, and we say that $\scrT$ is {\it finitely generated} if $\scrT$ admits a split generator.
\end{dfn}

\noindent
We need the following well-known Morita theorem for triangulated categories.
\begin{thm}[{\cite{kel2}}]\label{thm:keller equiv}
Let $\scrT$ be a dg-enhanced triangulated category.
\begin{itemize}
\item[$(1)$] Assume that $\scrT$ has a split generator $G\in\scrT$, and write $A\defeq \REnd(G)$ for the dg-endomorphism algebra. If $\scrT$ is idempotent complete, there is a triangulated equivalence 
\[
\scrT\cong \Perf A.
\]
\item[$(2)$] Assume $\scrT$ admits an exceptional sequence  $E_1,\hdots,E_n$, which is full, i.e. 
$\scrT=\lbr E_1,\hdots,E_n\rbr$.
Then $\scrT$ is idempotent complete, and there is a triangulated equivalence
\[
\scrT\cong \Perf \REnd(\oplus_{i=1}^nE_i).
\]
\end{itemize}
\begin{proof}
(1) This is \cite{kel2} (see also \cite[Theorem 3.8\,(b)]{kel3}). \\
(2)  The idempotent completeness follows from \cite[Corollary A.12]{ls} (see also \cite[Lemma 4.8]{bdfik}), since $\lbr E_i\rbr\cong \Db(k)$ is idempotent complete.
\end{proof}
\end{thm}

Let $F\colon \scrT\to \scrT'$ be an exact functor between  triangulated categories. For $\scrU\in\Th(\scrT)$, we define a full subcategory $F(\scrU)\defeq\{A\in\scrT'\mid A\cong F(B) \mbox{ for some } B\in\scrT\}\subseteq \scrT'$. Then the assignment $\scrU\mapsto \lbr F(\scrU)\rbr$ defines an order-preserving map
\[
\lbr F\rbr\colon \Th(\scrT)\to \Th(\scrT').
\]
Dually, for $\scrU'\in\Th(\scrT')$,
put $F^{-1}(\scrU')\defeq\{A\in \scrT\mid F(A)\in \scrU'\}$. Then $F^{-1}(\scrU')$ is  a thick subcategory of $\scrT$, and so there is an order-preserving map
\[
F^{-1}\colon \Th(\scrT')\to \Th(\scrT).
\]

\begin{prop}[{\cite[Lemma 3.1]{takahashi2}}]\label{verdier}
Let $\scrU\in \Th(\scrT)$, and denote by $F\colon \scrT\to \scrT/\scrU$ the natural quotient functor. Then the map $F^{-1}\colon \Th(\scrT/\scrU)\to \Th(\scrT)$ induces a lattice isomorphism
 \[
 F^{-1}\colon \Th(\scrT/\scrU)\simto \{\scrV\in \Th(\scrT)\mid \scrU\subseteq \scrV\},
 \]
 and its inverse is given by the assignment $\scrV\mapsto \scrV/\scrU$. 
\end{prop}

A functor $F\colon \scrT\to \scrT'$ is {\it dense} if every object in $\scrT'$ is a direct summand of $F(A)$ for some $A\in \scrT$. For  fully faithful  dense $F\colon \scrT\to \scrT'$ and  $\scrU\in\Th(\scrT)$, we define 
\[
\widetilde{F}(\scrU)\defeq\{A\in \scrT'\mid \exists B\in \scrT' \mbox{ such that } A\oplus B\in F(\scrU)\}.
\]

\begin{prop}\label{dense}
Let $F\colon \scrT\to\scrT'$ be a fully faithful dense functor. 
\begin{itemize}
\item[$(1)$] $\widetilde{F}(\scrU)$ is a thick subcategory of $\scrT'$. In particular, $\widetilde{F}(\scrU)=\lbr F(\scrU)\rbr$.
\item[$(2)$] The map $F^{-1}\colon \Th(\scrT')\to\Th(\scrT)$ is a lattice isomorphism, and its inverse is $\lbr F\rbr\colon \Th(\scrT)\to\Th(\scrT')$.
\end{itemize}
\begin{proof}
(1) By the same argument as in the proof of \cite[Proposition 3.13]{balmer},  the following equality holds:
\[
\widetilde{F}(\scrU)=\{A\in \scrT'\mid A\oplus (A[1])\in F(\scrU)\}.
\]
By this equality and the fully faithfulness of $F$, it is easy to see that $\widetilde{F}(\scrU)$ is a triangulated subcategory of $\scrT$. By definition, the subcategory $\widetilde{F}(\scrU)$ is closed under direct summands, and so $\widetilde{F}(\scrU)\in \Th(\scrT')$. {Then $\widetilde{F}(\scrU)=\lbr F(\scrU)\rbr$ holds.}

(2) The first assertion is \cite[Proposition 2.11\,(1)]{matsui}, and the latter one follows from (1) and a similar argument as in  \cite[Proposition 3.13]{balmer}.
\end{proof}
\end{prop}

 Thick subcategories $\scrU_1,\hdots, \scrU_n\in \Th(\scrT)$ are {\it orthogonal} to each other, denoted by $\scrU_1\perp\cdots\perp\scrU_n$, if $\Hom(A_i,A_j)=0$ for each $A_i\in \scrU_i$ and $i\neq j$. The triangulated category $\scrT$ is the {\it direct sum} of $\scrU_1,\hdots, \scrU_n$, denoted by \[\scrT=\scrU_1\oplus\cdots\oplus\scrU_n,\]
if $\scrU_1\perp\cdots\perp\scrU_n$  and every object of $\scrT$ is the direct sum of objects in $\scrU_i$. 
We say that $\scrT$ is {\it indecomposable} if there are no non-trivial thick subcategories $\scrU_1$ and $\scrU_2$ such that $\scrT=\scrU_1\oplus \scrU_2$. 
The following is standard, and so we omit the proof.

\begin{prop}\label{prop:direct sum}
Assume that $\scrT=\scrU_1\oplus\cdots \oplus\scrU_n$, and let $A\in \scrT$. Then, for each $1\leq i\leq n$, there exists $\pr_i(A)\in \scrU_i$, which is unique up to isomorphism,  such that $A\cong \pr_1(A)\oplus\cdots \oplus \pr_n(A)$. Furthermore, the assignment $A\mapsto \pr_i(A)$ defines an exact functor $\pr_i\colon \scrT\to \scrU_i$ that is right and left adjoint to the natural inclusion $\scrU_i\hookto \scrT$. In particular, $\scrU_i$ is an admissible subcategory of $\scrT$.
\end{prop}

By the previous proposition, there is a natural exact equivalence
\begin{equation}\label{direct sum quotient}
(\scrU_1\oplus\cdots\oplus\scrU_n)/\scrU_n\cong \scrU_1\oplus\cdots\oplus\scrU_{n-1}.
\end{equation}
The following is also elementary.

\begin{prop}\label{prop:orthogonal vee}
If $\scrU_1,\hdots,\scrU_n\in\Th(\scrT)$ are orthogonal to each other, then the full subcategory 
$\sum_{i=1}^n\scrU_i\defeq \{\oplus_{i=1}^nA_i\mid A_i\in \scrU_i \}\subseteq \scrT$ is a thick subcategory of $\scrT$. In particular, 
\[
\scrU_1\vee\cdots\vee\scrU_n=\scrU_1\oplus\cdots\oplus \scrU_n,
\]
where we identify $\scrU_i$ with a thick subcategory of the left hand side $\vee_{i=1}^n\scrU_i$.

\end{prop}

For finitely many posets $L_1,\hdots,L_n$, we define the {\it direct sum} of $L_1,\hdots,L_n$ to be the set $L_1\oplus\cdots\oplus L_n\defeq\{(a_1,\hdots,a_n)\mid a_i\in L_i\}$ with the order given by
\[
(a_1,\hdots,a_n)\leq(b_1,\hdots,b_n) \Longleftrightarrow a_i\leq b_i \mbox{ for }1\leq \forall i\leq n.
\]

\begin{prop}\label{prop:od}
Let  $\scrT=\scrT_1\oplus\cdots\oplus \scrT_n$ be a direct sum decomposition. 
\begin{itemize}
\item[$(1)$] For any $\scrU\in\Th(\scrT)$, we have $\pr_i(\scrU)=\scrT_i\cap \scrU$, and there is a direct sum decomposition $\scrU=\pr_1(\scrU)\oplus\cdots\oplus \pr_n(\scrU)$. 
\item[$(2)$]
The map
\[
f\colon \Th(\scrT)\simto \Th(\scrT_1)\oplus\cdots\oplus \Th(\scrT_n)
\]
given by $f(\scrU)\defeq (\pr_1(\scrU),\hdots,\pr_n(\scrU))$ is order-preserving and bijective.

\end{itemize}
\begin{proof}
(1)  follows from Proposition \ref{prop:direct sum}, and (2) follows from By Proposition \ref{prop:orthogonal vee}.
\end{proof}
\end{prop}

\begin{dfn} A triangulated category $\scrT$ is  {\it simple} if $\scrT\neq0$ and $\Th(\scrT)=\{0, \scrT\}$, and it is {\it semi-simple} if $\scrT$ is the direct sum of finitely many simple thick subcategories.
\end{dfn}

In what follows, we provide examples of semi-simple triangulated categories. Let $X$ be a noetherian scheme. For a specialization-closed subset $W$ of $X$,   $\Perf_WX$ (resp. $\D^{\rm b}_W(X)$) denotes the thick subcategory of $\Perf X$ (resp. $\Db(X)$) consisting of objects supported on $W$. Here 
a subset $W\subset X$ of a topological space $X$ is said to be {\it specialization-closed} if it is the union of (possibly infinitely many) closed subsets of $X$. The following classifies thick subcategories of perfect complexes on a noetherian ring.

\begin{thm}[{\cite[Theorem 1.5]{neeman}}]\label{class ring} Let $R$ be a noetherian ring. The map
\[
f\colon\{W\subseteq \Spec R\mid \mbox{ $W$ is specialization-closed }\}\to \Th(\Perf R) 
\]
given by $f(W)\defeq \Perf_WX$ is an order-preserving bijection.
\end{thm}

In order to provide examples of semi-simple triangulated categories, we need the following lemmas.

\begin{lem}[{\cite{balmer3}}]\label{disjoint union} Let $X$ be a noetherian separated  scheme, and let $Z_1$ and $Z_2$ be two  disjoint closed subsets in $X$. Then $\Perf_{Z_1\sqcup Z_2}X=\Perf_{Z_1}X\oplus \Perf_{Z_2}X$. 
\begin{proof}
The orthogonality follows from \cite[Proposition 4.1, Corollary 2.8]{balmer3}, and the direct sum follows from \cite[Theorem 2.11]{balmer3}.
\end{proof}
\end{lem}

The first assertion of the following result is due to Balmer \cite[Theorem 2.13]{balmer2}, and the second one follows from the first one and Proposition \ref{dense}\,(2) due to \cite[Proposition 2.11\,(1)]{{matsui}}.

\begin{lem}[{\cite{balmer2,matsui}}]\label{open rest} Let $X$ be a noetherian   scheme, and let $U\subseteq X$ be an open subset. Then the restriction functor $(-)|_U\colon \Perf X\to \Perf U$ induces a fully faithful dense functor
\[
(-)|_U\colon\Perf X/\Perf_{X\backslash U}X\hookto \Perf U.
\]
In particular, there is a bijection $\Th(\Perf X/\Perf_{X\backslash U}X)\cong \Th(\Perf U)$.
\end{lem}

The following result shows that $\Perf X$ for a noetherian separated scheme $X$ always contains  a semi-simple thick subcategory.

\begin{prop}\label{prop:torsion simple}
Let $X$ be a noetherian separated scheme.

\begin{itemize}  
\item[$(1)$] For a closed point $p\in X$, the thick subcategory $\Perf_p X$  is simple.

\item[$(2)$] For a finite set $Z=\{p_1,\hdots,p_n\}\subseteq X$ of $n$-distinct closed points,  $\Perf_ZX=\Perf_{p_1}X\oplus\cdots\oplus \Perf_{p_n}X$ holds, and in particular $\Perf_ZX$ is semi-simple. 
\end{itemize}
\begin{proof}
(1)  Take a non-zero  thick subcategory $\scrS\neq0$ of $\Perf_pX$. Let $p\in U=\Spec R$ be an open affine neighborhood of $p$, and set $Z\defeq X\backslash U$. By Lemma \ref{disjoint union},  $\Perf_pX\perp \Perf_ZX$ holds, and so $\Perf_pX\vee \Perf_ZX\cong \Perf_pX\oplus \Perf_ZX$ and $\scrS\vee \Perf_ZX=\scrS\oplus \Perf_ZX$. By Proposition \ref{verdier}, Proposition \ref{dense} and Lemma \ref{open rest}, there is a bijective map
\[
\Phi\colon \{\scrV\in \Th(\Perf X)\mid \Perf_ZX\subseteq \scrV\}\simto \Th(\Perf U)
\]
given by 
\[
\Phi(\scrV)=\{A\in \Perf U\mid \exists B\in \Perf U \mbox{ such that } A\oplus B\in (\scrV/\Perf_ZX)|_U\}.
\]
It is easy to see that $\Phi(\Perf_pX\oplus \Perf_ZX)\subseteq \Perf_pU$. Since the bijection $\Phi$ is order-preserving, the inclusions
\[
0\neq \Phi(\scrS\oplus \Perf_ZX)\subseteq \Phi(\Perf_pX\oplus \Perf_ZX)\subseteq \Perf_pU
\]
hold. By Theorem \ref{class ring}, $\Perf_pU$ is simple. Hence $\scrS\oplus \Perf_ZX=\Perf_pX\oplus \Perf_Z X$ holds, and by taking Verdier quotients by $\Perf_ZX$, we obtain $\scrS=\Perf_pX$ by \eqref{direct sum quotient}. This shows that $\Perf_pX$ is simple.\\
(2) The first assertion follows from Lemma \ref{disjoint union}, and the second one follows from the first one and (1).
\end{proof}
\end{prop}

In the rest of this subsection, we assume that  for two objects $A,B\in\scrT$, the $k$-vector space $\Hom(A,B)$ is finite-dimensional. Furthermore, we assume that $\scrT$ admits a dg-enhancement and that it is idempotent complete. In this setting, we prove the simplicity of a thick subcategory generated by a bouquet sphere-like object. 

\begin{dfn}\label{dfn:bouquet}
Let $S\in \scrT$, $n>0$ and $d\neq0\in\bZ$. We say that $S$ is {\it $n$-bouquet $d$-sphere-like} if equations 
\[
\Hom(S,S[i])=
\begin{cases}
k & i=0\\
0 & i\neq0,d\\
k^n &i=d
\end{cases}
\]
hold. A $1$-bouquet $d$-sphere-like object is called a  {\it $d$-sphere-like} object. If we do not need to specify the integers, we just call it a {\it bouquet sphere-like} object. 
\end{dfn}

\begin{rem}\label{rem:bouquet}
Note that, if $d>0$, the graded vector space $\Hom(B,B[*])$ of an $n$-bouquet $d$-sphere-like object $B$ is isomorphic to the total singular cohomology $H^*(\vee_n S^d,k)$ of the bouquet of $n$ spheres $S^d$ with coefficients $k$. 
\end{rem}

\begin{prop}\label{prop:bouquet}
 Let $S\in \scrT$ be a bouquet sphere-like object. The thick subcategory $\lbr S\rbr$ generated by $S$ is simple.
\begin{proof}
Consider the graded endomorphism algebra \[B\defeq \bigoplus_{m\in\bZ}\Hom(S,S[m])\] of $S$.
Denote by $A\defeq\REnd(S)$ the dg-endomorphism algebra of $S$ -- it exists since we assume that $\scrT$ admits a dg-enhancement. Then there is an equivalence $\lbr S\rbr\cong \Perf A$ by Theorem \ref{thm:keller equiv}\,(1).
The cohomology algebra $H^*(A)$ is isomorphic to $B$ as graded $k$-algebras. By assumption on $S$, the algebra $B$ is concentrated in degrees $0$ and $d \neq 0$. It follows as in \cite[Theorem 2.1]{kyz} that $A$ is quasi-isomorphic to $B$.
Hence, there is a sequence of equivalences
\[
\lbr S\rbr\cong \Perf A\cong \Perf B.
\]
Since the graded ring $B$ is commutative, connected, and finite-dimensional over $k$, there is a unique homogeneous prime ideal of $B$. Thus, $\Perf B$ is simple by \cite[Theorem A.2]{bw}, and so is $\lbr S\rbr$.
\end{proof}
\end{prop}

\vspace{2mm}
\subsection{Prime and maximal thick subcategories}~

In this section, we recall the definitions and fundamental properties of prime thick subcategories introduced by Matsui \cite{matsui, matsui2}.

For a thick subcategory $\scrP$ of $\scrT$ we define
\begin{align*}
\Th(\scrT)_{>\scrP}&\defeq\{\scrQ\in \Th(\scrT)\mid \scrP\subsetneq\scrQ\}\subsetneq \Th(\scrT)\\
\Th(\scrT)_{\geq\scrP}&\defeq\{\scrQ\in \Th(\scrT)\mid \scrP\subseteq\scrQ\}\subseteq \Th(\scrT)
\end{align*}

\begin{dfn}[\cite{matsui2}]\label{dfn:prime}
A thick subcategory $\scrP$ is {\it prime} if the set $\Th(\scrT)_{>\scrP}$ has a minimum.  We denote by $\spec\scrT$ the set of prime thick subcategories of $\scrT$, and call it the {\it Matsui spectrum} of $\scrT$. 
\end{dfn}

For $\scrP\in\spec\scrT$, we call the minimum in $\Th(\scrT)_{>\scrP}$ the {\it cover} of $\scrP$, and denote it by $\overline{\scrP}$.
Note that $\overline{\scrP}=\bigcap_{\scrQ\in\Th(\scrT)_{>\scrP}}\scrQ$.

\begin{dfn} A thick subcategory $\scrM\in\Th(\scrT)$ is said to be {\it maximal} if the quotient $\scrT/\scrM$ is simple. We denote by $\Max(\scrT)$ the set of maximal thick subcategories of $\scrT$.
\end{dfn}

Note that $\scrM\in\Th(\scrT)$ is maximal if and only if  any $\scrU\in\Th(\scrT)$ with $\scrM\subseteq \scrU\subseteq\scrT$ must be  either $\scrM$ or $\scrT$. Thus for $\scrM\in\Max(\scrT)$,  $\Th(\scrT)_{>\scrM}=\{\scrT\}$ holds, and in particular $\Max(\scrT)\subseteq \spec\scrT$. The following ensures the existence of a maximal thick subcategory in a triangulated category with a split generator.

\begin{prop}\label{prop:matsui spec nonempty}
Assume that $\scrT\neq 0$ admits a split generator $G\in\scrT$. For any thick subcategory $\scrN\subsetneq\scrT$, there is a maximal thick subcategory $\scrM$  of $\scrT$ such that  $\scrN\subseteq \scrM$. In particular, $\spec\scrT\neq \emptyset$.
\begin{proof}
Set $\Sigma\defeq \Th(\scrT)_{\geq \scrN}\backslash \{\scrT\}$.
By Zorn's lemma, it suffices to show that for any non-empty totally ordered subset $\Lambda\subseteq \Sigma$, there is an upper bound of $\Lambda$ in $\Sigma$. Since $\Th(\scrT)$ is a complete lattice, there exists the supremum $\vee_{\scrU\in \Sigma}\scrU$ of $\Sigma$ in $\Th(\scrT)$. We claim that $\vee_{\scrU\in \Sigma}\scrU\in\Sigma$. It is obvious that $\vee_{\scrU\in \Sigma}\scrU\in\Th(\scrT)_{\geq \scrN}$. Since $\Sigma$ is tortally ordered, $\vee_{\scrU\in \Sigma}\scrU=\cup_{\scrU\in \Sigma}\scrU$ holds. This implies that $\vee_{\scrU\in \Sigma}\scrU\neq \scrT$, since $\scrT$ has a split generator. Thus $\vee_{\scrU\in \Sigma}\scrU\in \Sigma$.
\end{proof}
\end{prop}

\begin{rem}
By \cite[Corollary 2.10]{matsukawa}, $\spec\scrT \neq \emptyset$ for arbitrary triangulated categories $\scrT\neq 0$.
\end{rem}

Let $X$ be a smooth projective variety. We say that a non-zero admissible subcategory $\scrA$ of $\Db(X)$ is  a {\it phantom subcategory} if ${\rm HH}_*(\scrA)=K_0(\scrA)=0$. The following shows that the Matsui spectrum of a phantom subcategory is not empty.

\begin{cor}\label{cor:matsui spec admissible}
Let $X$ be a smooth projective variety, and let $\scrA$ be a non-zero admissible  subcategory of $\Db(X)$. Then $\spec\scrA\neq \emptyset$.
\begin{proof}
Since  $\Db(X)$ admits a split generator \cite{bv,rou}, so does $\Db(X)/\scrA^{\perp}\cong \scrA$. Therefore the result follows from Proposition \ref{prop:matsui spec nonempty}.
\end{proof}
\end{cor}

\noindent
Let $X$ be a noetherian scheme. A thick subcategory $\scrI$ of $\Perf X$ is an {\it ideal} if for objects $A\in\Perf X$ and $I\in\scrI$, $A\otimes_XI\in\scrI$ holds. An ideal $\scrP$ of $\Perf X$ is said to be {\it prime} if for objects $A,B\in\Perf X$ the condition $A\otimes_X B \in \scrP$ implies that $A\in\scrP$ or $B\in\scrP$. We denote by $\Spec_{\otimes}\Perf X$ the set of prime ideals of $\Perf X$. For every point $x\in X$, consider the full subcategory given by
\[
\scrS_X(x)\defeq\{F\in\Perf X\mid x\not\in\Supp(F)\}.
\]

\begin{thm}\cite{balmer}\label{thm:balmer}
For every $x\in X$, $\scrS_X(x)$ is a prime ideal. Moreover, the assignment $x\mapsto \scrS_X(x)$ defines a bijective map $X\simto \Spec_{\otimes} \Perf X$. 
\end{thm}

The following shows that prime thick subcategories are generalizations of prime ideals. 

\begin{thm}[{\cite[Corollary 4.9]{matsui}}]\label{thm:prime ideal} Let $\scrP$ be an ideal of $\Perf X$. Then $\scrP$ is a prime ideal if and only if it is a prime thick subcategory.
\end{thm}

For a thick subcategory $\scrU$ of $\Perf X$, we define 
\[
\Supp(\scrU)\defeq \bigcup_{A\in\scrU} \Supp(A)\subseteq X.
\]
For later use, we prove the following.
\begin{lem}\label{lem:fg closed}
If $\scrU\in\Th(\Perf X)$ is finitely generated, $\Supp(\scrU)$ is closed.
\begin{proof}
If  $\scrU=\lbr G\rbr$ for some $G\in\scrU$,  it follows that $\Supp(\scrU)=\Supp(G)$. Hence $\Supp(\scrU)$ is closed.
\end{proof}
\end{lem}

\vspace{2mm}
\subsection{Composition series, length and JD property}
~

\begin{dfn}\label{D:CS}
A {\it composition series}  of $\scrT$ is a finite sequence 
\[
\scrS_{*}=(0=\scrS_0\subsetneq \scrS_1\subsetneq\cdots\subsetneq \scrS_n=\scrT)
\]
of $\scrS_i\in\Th(\scrT)$ such that for each $1\leq i\leq n$, the quotient $\scrS_i/\scrS_{i-1}$ is simple. We denote by $\CS(\scrT)$ the set of composition series in $\scrT$. If $\scrS_*=(\scrS_0\subsetneq\cdots\subsetneq\scrS_n)\in \CS(\scrT)$, we set $\ell(\scrS_*)\defeq n$, and call it the {\it length} of $\scrS_*$.
\end{dfn}

\begin{rem} \label{rem:alk}
In \cite{alk} another version of composition series for derived module categories are studied. However, they only allow thick subcategories that are both admissible and equivalent to derived module categories. In particular, categories that are simple in the setting of \cite{alk} will typically not be simple in our setting and not every full exceptional sequence gives rise to a composition series in their setting, cf. also \cite{Kalckgentle}.  
\end{rem}

 If $\scrT$ admits a composition series, we set 
\[
\ell(\scrT)\defeq\min\left\{\ell(\scrS_*)\relmiddle|\mbox{$\scrS_*\in \CS(\scrT)$}\right\},
\]
and call it the {\it length} of $\scrT$. If $\scrT=0$, we put $\ell(\scrT)\defeq0$, and if $\scrT$ does not admit any composition series, we set $\ell(\scrT)\defeq\infty$. By definition, $\ell(\scrT)=1$ if and only if $\scrT$ is simple. Moreover, $\ell(\scrT)=2$ if and only if there is $\scrS_*\in\CS(\scrT)$ with $\ell(\scrS_*)=2$.

\begin{prop}\label{prop:fg max}
Assume that $\ell(\scrT)<\infty$. 
\begin{itemize}
\item[$(1)$] $\scrT$ is finitely generated.
\item[$(2)$] There is a finitely generated maximal thick subcategory of $\scrT$.
\end{itemize}
\begin{proof}
If $\ell(\scrT)<\infty$, there is a composition series $\scrS_{*}=(\scrS_0\subsetneq\cdots\subsetneq\scrS_n)$ in $\scrT$. Since $\scrS_1$ is simple, $\scrS_1=\lbr A_1\rbr$ holds for every non-zero object $A_1\in\scrS_1$. Since there is no  thick subcategory $\scrU$ with $\scrS_1\subsetneq \scrU\subsetneq \scrS_2$, $\scrS_2=\lbr A_1,A_2\rbr$ holds for any object $A_2\in\scrS_2\backslash \scrS_1$. Repeating this argument shows that  $\scrS_i$ is finitely generated for each $1\leq i\leq n$. In particular, $\scrT=\scrS_n$ is finitely generated, and $\scrS_{n-1}$ is a finitely generated maximal thick subcategory.
\end{proof}
\end{prop}

\noindent
Full exceptional sequences yield composition series in the following way.

\begin{prop}\label{prop:length exceptional}
If $\scrT$ admits a full exceptional sequence $E_1,\hdots,E_n$, then the sequence
\begin{equation}\label{eq:fec}
\Bigl( \lbr E_1\rbr\subset \lbr E_1,E_2\rbr\subset\cdots\subset\lbr E_1,\hdots,E_n\rbr\Bigr)
\end{equation}
forms a composition series of $\scrT$. In particular, 
$\ell(\scrT)\leq \rk(K_0(\scrT))$.
\begin{proof}
Set $\scrS_i\defeq \lbr E_1,\hdots,E_i\rbr$. Then $\scrS_i/\scrS_{i-1}\cong \lbr E_i\rbr$ holds, and each $\lbr E_i\rbr\cong \Db(k)$ is simple. This proves the first assertion. The second one follows since $K_0(\scrT)=\bigoplus_{i=1}^n[E_i]\cong \bZ^n$.
\end{proof}
\end{prop}

We don't know any example of a triangulated category $\scrT$  with a full exceptional sequence and  $\ell(\scrT)<\rk(K_0(\scrT))$.  So we ask the following.

\begin{ques}\label{Q:FESminimal}
Does the existence of a full exceptional sequence in $\scrT$ imply $\ell(\scrT)=\rk(K_0(\scrT))$?
\end{ques}

The following two propositions are immediate consequences of Propositions \ref{dense} and  \ref{prop:od} respectively, and so we omit the proofs.
\begin{prop}
Let $F\colon\scrT\to \scrT'$ be a  fully faithful dense functor. For $\scrS'_*=(\scrS'_0\subsetneq\cdots\subsetneq\scrS'_n)\in \CS(\scrT')$, we have a composition series $F^{-1}(\scrS'_*)\in \CS(\scrT)$ given by 
\[
F^{-1}(\scrS'_*)\defeq (F^{-1}(\scrS'_0)\subsetneq\cdots\subsetneq F^{-1}(\scrS'_n)).
\]
Furthermore, this defines a bijective map 
\[F^{-1}(-)\colon \CS(\scrT')\simto \CS(\scrT)\]
that preserves length. In particular, $\ell(\scrT)=\ell(\scrT')$.
\end{prop}

\begin{prop}\label{prop:semi-simple length}
Let $\scrT=\scrT_1\oplus\cdots\oplus \scrT_n$ be a direct sum decomposition. Then $\ell(\scrT)=\sum_{i=1}^n\ell(\scrT_i)$. In particular, if each $\scrT_i$ is simple, $\ell(\scrT)=n$.
\end{prop}

The following observation is useful.

\begin{lem}\label{lem:additivity}
Let $\scrU\in \Th(\scrT)$. If there are $\scrS_*\in\CS(\scrU)$ and $\scrS'_*\in\CS(\scrT/\scrU)$, then there exists $\widetilde{\scrS}_*\in\CS(\scrT)$ with $\ell(\widetilde{\scrS}_*)=\ell(\scrS_*)+\ell(\scrS'_*)$.
\end{lem}

\begin{dfn}\label{dfn:composite}
We say that a thick subcategory $\scrU\in \Th(\scrT)$ is {\it composite} in $\scrT$ if $\ell(\scrU)<\infty$ and $\ell(\scrT/\scrU)<\infty$. 
\end{dfn}

\begin{rem}
Note that $\scrU\in\Th(\scrT)$ is composite in $\scrT$ if and only if there is a composition series $\scrS_*=(\scrS_i)_{0\leq i\leq n}\in\CS(\scrT)$ in $\scrT$ such that $\scrU=\scrS_j$ for some $0\leq j\leq n$. By definition,  the following are equivalent: 
\begin{itemize}
\item[(1)] $\scrT$ is composite in $\scrT$.
\item[(2)] The trivial subcategory $0$ is composite in $\scrT$.
\item[(3)] $\ell(\scrT)<\infty$.

\end{itemize}
\end{rem}

The following statement says that taking length satisfies subadditivity. 

\begin{prop}\label{prop:length verdier}
Let $\scrU\in \Th(\scrT)$. If $\scrU$ is composite, then 
\[
\ell(\scrT)\leq \ell(\scrU)+\ell(\scrT/\scrU)<\infty.
\]
\begin{proof}
This follows from Lemma \ref{lem:additivity}.
\end{proof}
\end{prop}

\begin{rem}
(1) Let $\scrU,\scrV\in\Th(\scrT)$. Even if $\ell(\scrT)<\infty$, neither $\scrU$ nor $\scrT/\scrV$ admits any composition series in general. For example, by Proposition \ref{prop:length exceptional}, we see that $\ell(\Db(\bP^1))=2<\infty$. However, $\scrU\defeq \{F\in \Db(\bP^1)\mid \Supp(F)\neq \bP^1\}$ is not finitely generated, and in particular $\ell(\scrU)=\infty$. Let $p\defeq (1:0)\in \bP^1$, and set $U\defeq \bP^1\backslash \{p\}\cong \bA^1$. Then $\scrV\defeq\Perf_p\bP^1$ is a thick subcategory of $\Db(\bP^1)$, and by Lemma \ref{open rest}, there is a bijection  $\Th(\Db(\bP^1)/\scrV)\cong \Th(\Db(\bA^1))$. 
Hence Corollary \ref{cor:commutative ring} below shows that the quotient $\Db(\bP^1)/\scrV$ does not admit any composition series.\\
(2) The perfect derived category $\Perf k Q$ of a graded Kronecker quiver $Q$ contains a simple thick subcategory $\scrU$ with $\ell(\Perf k Q/\scrU)=\ell(\Perf k Q)=2$ (see Section \ref{section:rational surfaces} for the details).
This shows that the inequality in Proposition \ref{prop:length verdier} is strict in general, even if $\ell(\scrU)<\infty$ and $\ell(\scrT/\scrU)<\infty$.
\end{rem}

Lemma \ref{lem:additivity} implies the following.

\begin{prop}\label{prop:length sod}
Let $\scrT=\l\scrA_1,\hdots,\scrA_n\r$ be a semi-orthogonal decomposition with $\ell(\scrA_i)<\infty$ for each $1\leq i\leq n$. If $\scrS^i_*\in\CS(\scrA_i)$, there exists $\scrS_*\in\CS(\scrT)$ with $\ell(\scrS_*)=\sum_{i=1}^n\ell(\scrS^i_*)$. In particular,   $\ell(\scrT)\leq \sum_{i=1}^n\ell(\scrA_i)$ holds.
\end{prop}

We consider the following property.

\begin{dfn}
For a finite length $\scrT$, we say that $\scrT$ satisfies the {\it JD property} if $\ell(\scrS_*)=\ell(\scrS'_*)$ holds  for any $\scrS_*,\scrS'_*\in\CS(\scrT)$.
\end{dfn}

After a draft version of this paper appeared, Greg Stevenson and the second named author independently noticed that derived categories of certain finite-dimensional algebras do not satisfy the JD property. 
We generalize this observation in Theorem \ref{thm:appendix 2}, and use it in section \ref{section:rational surfaces} to obtain geometric examples without the JD property.

\begin{exa} \label{Ex:JDP}
By Proposition \ref{prop:od}, if $\cT=\cT_1 \oplus \cdots \oplus \cT_\ell$ and all $\cT_i$ satisfy the JD property, then $\cT$ satisfies the JD property.
In particular, all semi-simple triangulated categories satisfy the JD property.
\end{exa}

Hiroki Matsui pointed out the following: let $X$ be a topological space, and let  $\Spcl(X)$ be the set of specialization-closed subsets of $X$. Consider the similar notions of composition series, length and the JD property for the poset $\Spcl(X)$.

\begin{lem}\label{lem:matsui}
If $X$ is a $T_0$-space, the following are equivalent. 
\begin{itemize}
\item[$(1)$] $\Spcl(X)$ has a composition series 
\item[$(2)$] The set $X$ is  finite. 
\end{itemize}
Furthermore, if these conditions hold, $\ell(\Spcl(X))=\#X$ holds, and  $\Spcl(X)$ satisfies the JD property. 
\begin{proof}
(1) $\Rightarrow$ (2) Let $W_1,W_2\in \Spcl(X)$ such that $W_1\subsetneq W_2$ and there is no $W\in\Spcl(X)$ with $W_1\subsetneq W \subsetneq W_2$. We claim that $W_2=W_1\cup\{x\}$ for a unique $x\in W_2\backslash W_1$. Indeed, let $x\in W_2\backslash W_1$. Since $X$ is $T_0$, the subset $\overline{\{ x\}}\backslash \{x\}$ is also specialization-closed. Consider the following chain in $\Spcl(X)$:
\[
W_1\subseteq W_1\cup \Bigl(\overline{\{ x\}}\backslash \{x\}\Bigr)\subsetneq W_1\cup \overline{\{ x\}}\subseteq W_2,
\]
where the last inclusion holds since $W_2$ is specialization-closed.
By the assumption, we have $W_1=W_1\cup \Bigl(\overline{\{ x\}}\backslash \{x\}\Bigr)$ and $W_1 \cup \overline{\{ x\}} = W_2$. These equalities imply the claim. If $\Spcl(X)$ has a composition series
\[
W_0=\emptyset \subsetneq W_1 \subsetneq\cdots\subsetneq W_n=X,
\]
the claim shows that $\#X=n$. (2) $\Rightarrow$ (1) is obvious, and the latter assertion follows from the above argument.
\end{proof}
\end{lem}

\begin{cor}\label{cor:commutative ring}
Let $R$ be a noetherian commutative ring. Then $\ell(\Perf R)<\infty$ if and only if the set $\Spec R$ is finite. In this case, $\Perf R$ satisfies the JD property. 
\begin{proof}
This follows from Theorem \ref{class ring} and Lemma \ref{lem:matsui}.
\end{proof}
\end{cor}

\vspace{2mm}
\subsection{Krah's phantom subcategory}\label{section:Krah's phantom}~

In this section, we explain that the study of the JD property might be useful for the study of phantom categories. 

Let $X$ be the blow-up of $\bP_{\bC}^2$ at $10$ general closed points. Then $\Db(X)$ has a full exceptional sequence, and $\rk(K_0(\Db(X)))=13$. 

\begin{thm}[{\cite[Theorem 1.1]{kra}}] Let $X$ be the blow-up of $\bP_{\bC}^2$ at $10$ general closed points.
Then $\Db(X)$ admits an exceptional sequence $E_1,\hdots,E_{13}$ that is not full. 
\end{thm}

The above implies that 
\[\scrP\defeq \lbr E_1,\hdots,E_{13}\rbr^{\perp}\]
is a phantom subcategory. By the following remark, if $\Db(X)$ satisfies the JD property,  $\ell(\scrP)=\infty$ holds, and in particular, $\scrP$ is not simple. 

\begin{rem}\label{R:2.42}
Let $\scrT$ be a triangulated category with $\ell(\scrT)<\infty$, and let $\scrU\in\Th(\scrT)\backslash\{\scrT\}$ such that there is  a composition series $\scrS_*\in\CS(\scrU)$ with $\ell(\scrS_*)\geq\ell(\scrT)$.    If $\scrT$ satisfies the JD property, then $\ell(\scrT/\scrU)=\infty$. 

 Indeed, if $\ell(\scrT/\scrU)<\infty$, we can extend $\scrS_*$ to obtain $\widetilde{\scrS}_*\in\CS(\scrT)$ with $\ell(\widetilde{\scrS}_*)>\ell(\scrS_*)\geq\ell(\scrT)$. This contradicts the JD property of $\scrT$.  
\end{rem}

\begin{rem}
In Section \ref{section:JD property}, we provide counterexamples to the JD property of derived categories of certain smooth projective rational surfaces. However, the surfaces appearing in these counterexamples contain curves $C$ with self-intersection number $C^2<-1$, and these curves induce composition series of different lengths. On the other hand, the surface $X$ from above does not contain such a curve, and this might allow $\Db(X)$ to satisfy the JD property, cf.\ also Corollary \ref{cor:admissibleJDP} and Remark \ref{rem:SODintoJDP} for first steps in this direction.
\end{rem}

As we mentioned above, if $\Db(X)$ satisfies the JD property, the phantom $\scrP$ is not simple. In the following, we give examples of non-simple phantom subcategories on smooth projective varieties.

\begin{exa}\label{exa:hilbert} 
As above, let $X$ be one of the surfaces studied by Krah. 
Let $Y=\Hilb^n(X)$ be the Hilbert scheme of $n$ points on $X$. Then by \cite[Theorem 3.4]{kos} (see also the proof of \cite[Lemma 4.4]{kos}), $\Db(Y)$ is semi-orthogonally decomposed into several copies of the symmetric products $\scrA_i\defeq\Sym^i(\scrP)$ of $\scrP$ ($0\leq i\leq n$), which are also phantom if $i>0$. Thus the admissible subcategory  $\lbr\,\scrA_i\,|\, i>0\,\rbr\subset \Db(Y)$ is a phantom subcategory which has a semi-orthogonal decomposition whose components are $\scrA_i$.

Further examples of phantoms admitting semi-orthogonal decompositions arise from projective bundles on $X$ by work of Orlov \cite{orlov4}.
\end{exa}

\begin{rem}\label{Rem:2.44(old)}
In contrast to the examples above, simple triangulated categories $\scrT$ with vanishing Grothendieck group $K_0(\scrT)$ are also known. Indeed, consider the orbit category $$\cC_Q\defeq\Db(kQ)/(\bS [-2])$$ for a quiver $Q$,  where $\bS$ is the Serre-functor. $\cC_Q$ is called the {\it cluster category} of $Q$. Then $\cC_Q$ is simple for $Q$ Dynkin by \cite[Thm 8.1]{koe} and has vanishing Grothendieck group for $Q$ of Dynkin type $\bA_{2n}$ and $\bE_6, \bE_8$ by \cite[Prop. 5]{bkm}.

Other examples are given by the singularity categories of simple curve singularities of  Dynkin type $\bA_{2n}$, $\bE_6$, $\bE_8$ and the simple surface singularity of Dynkin type $\bE_8$, cf.\ \cite[Chapter 13]{yoshino} for vanishing of the Grothendieck group and \cite{takahashi1} for the simplicity statement. 

Note that these singularity categories with vanishing Grothendieck groups have different features from phantom categories, since their Hochschild homology groups are non-zero \cite[Theorem 6.6]{dyc}. We don't know whether the cluster categories above  have vanishing Hochschild homology .
\end{rem}

\vspace{2mm}
\subsection{Length spectra}~

In this section, we always assume that $\ell(\scrT)<\infty$. Consider the set
\[
\LS(\scrT)\defeq \{\ell(\scrS_*)\mid \scrS_*\in\CS(\scrT)\}\subseteq \bN\cup \{\infty\}
\]
of the length of all composition series of $\scrT$, and we call it the {\it length spectrum} of $\scrT$. Moreover, we set
\[
\ellu(\scrT)\defeq \sup\LS(\scrT),
\]
and  we call it the {\it ultimate length} of $\scrT$. When $\ellu(\scrT)<\infty$, we define the {\it Jordan--Dedekind index} of $\scrT$ by
\begin{equation}\label{eqn:JD index}
\JD(\scrT)\defeq\ellu(\scrT)-\ell(\scrT),
\end{equation}
and when $\ellu(\scrT)=\infty$, we put $\JD(\scrT)\defeq \infty$. 
By definition, $\scrT$ satisfies the JD property if and only if $\CS(\scrT)\neq \emptyset$ and $\JD(\scrT)=0$. In section \ref{section:JD property}, we will see that for every positive integer $n$, there is an indecomposable triangulated category $\scrT$ such that $\JD(\scrT)\geq n$. In what follows, we list natural questions on the above invariants. The first one is on the boundedness of length spectra.

\begin{ques}
Is there a finite length $\scrT$ such that $\ellu(\scrT)=\infty$?
\end{ques}

Examples of triangulated categories with $\JD(\scrT)>0$, which we know, satisfies $\JD(\scrT)<\ell(\scrT)$, and so we pose the following question.

\begin{ques}
If $\JD(\scrT)<\infty$, does the inequality $\JD(\scrT)< \ell(\scrT)$ hold?
\end{ques}

Orlov introduced the notion of {\it Orlov spectrum} of $\scrT$, which is defined to be the set of generation times of all split generators of $\scrT$ \cite{orl2}. In \cite{bfk}, gaps in Orlov spectra are considered.  
Similarly to this, we ask the existences of gaps in length spectra.

\begin{ques}
Is there a finite length $\scrT$ such that its length spectrum has a gap, i.e., there exists an integer $\ell(\scrT)<n<\ellu(\scrT)$ with $n\not\in\LS(\scrT)$? 
\end{ques}

\vspace{3mm}
\section{Triangulated categories with the JD property}\label{section:example}

Throughout this section $k$ is assumed to be algebraically closed.  

\vspace{2mm}
\subsection{Hereditary algebras}~

 In this section, we prove that the derived category of a finite-dimensional hereditary\footnote{That is, the global dimension is at most $1$.} algebra satisfies the JD property.  To this end, we begin by  recalling basic properties of thick subcategories of derived categories of hereditary abelian categories. Here, an abelian category $\scrA$ is hereditary if $\Ext^i_{\scrA}(A,B)=0$ for all $A,B\in \scrA$ and all $i>1$. Recall that an exact abelian subcategory $\scrW$ of an abelian category $\scrB$ is wide if it is closed under extensions.

Let $\scrA$ be a hereditary abelian category over $k$.
For a collection $\cC$ of objects in $\scrA$, we denote by 
$\wide(\cC)$
the smallest wide subcategory of $\scrA$ containing $\cC$. For a wide subcategory $\scrW\subseteq\scrA$, we define
\[
\D^{\rm b}_{\scrW}(\scrA)\defeq \{F\in \Db(\scrA)\mid H^i(F)\in \scrW \mbox{ for all } i\in \bZ\}.
\]
One can check that $\D^{\rm b}_{\scrW}(\scrA)=\lbr\scrW\rbr\subseteq \Db(\scrA)$ and $\D^{\rm b}_{\scrW}(\scrA)\cong \Db(\scrW)$.
The following is standard.

\begin{prop}\label{prop:wide thick} Notation is the same as above.
\begin{itemize}
\item[$(1)$] There is a bijective correspondence 
\[
\Th(\Db(\scrA))\longleftrightarrow\left\{ \mbox{wide subcategory of } \scrA\right\}
\]
that preserves inclusions. The map from the left to the right is given by $\scrU\mapsto \scrU\cap \scrA$, and the inverse is given by $\scrW\mapsto \D^{\rm b}_{\scrW}(\scrA)=\lbr\scrW\rbr$.
\item[$(2)$] Let $\cC\subset \scrA$ be a collection of objects in $\scrA$. Then 
\[\lbr\wide(\cC)\rbr=\lbr \cC\rbr
 \hspace{3mm}\mbox{ and }\hspace{3mm} \lbr \cC\rbr\cap \scrA=\wide(\cC).
\]
\end{itemize}
\begin{proof}
(1) is \cite[Theorem 5.1]{bru}, and  (2) follows from (1).
\end{proof}
\end{prop}

{By work of Gabriel, every finite-dimensional hereditary algebra $\Lambda$ over $k$ is Morita equivalent to the path algebra $kQ$ of a finite acyclic quiver $Q$. Since we are only interested in (derived) module categories, in the following subsection, we can always work with path algebras $kQ$.}

Let $Q$ be a finite acyclic quiver with $n$ vertices. Then the abelian category $\fmod kQ$ is hereditary. Since the indecomposable projective $kQ$-modules corresponding to vertices form a full exceptional sequence of $\Db(kQ)$, we have $\ell(\Db(kQ))\leq n<\infty$. We recall  the notion of semibricks.

\begin{dfn}
An  object $B\in\fmod kQ$ is called a {\it brick} if $\End_{kQ}(B)\cong k$. Denote by $\brick kQ$ the set of isomorphism classes of bricks in $\fmod kQ$.
A finite subset $\cS=\{B_1,\hdots,B_r\}\subset \brick kQ$ is called a {\it semibrick} in $\fmod kQ$ if $\Hom(B_i,B_j)=0$ for $i\neq j$.
\end{dfn}

Since $\fmod kQ$ is hereditary, the lattice $\Th(\Db(kQ))$ is isomorphic to the lattice of wide subcategories of $\fmod kQ$. 
Furthermore, by {\cite[Section 1.2]{ringel}, finitely generated wide subcategories of $\fmod kQ$  correspond to semibricks in $\fmod kQ$. Therefore, the following holds.

\begin{prop}[cf.\ {\cite[Proposition 3.10]{el}}]\label{prop:unique semibrick}
For a finitely generated thick subcategory $\scrU\in \Th(\Db(kQ))$, there exists a unique semibrick $\cS\subset \brick kQ$ such that $\scrU=\lbr\cS\rbr$. 
\end{prop}

We say that a module $E\in \fmod kQ$ is {\it exceptional} if it is exceptional as an object in $\Db(\fmod kQ)$. For an exceptional module $E\in\fmod kQ$, we define 
\[E_{\fmod}^{\perp}\defeq\{ M\in\fmod kQ\mid \Ext^i(E,M)=0 \mbox{ for } i=0,1\}\subsetneq \fmod kQ.\]
One can check that $E_{\fmod}^{\perp}$ is a wide subcategory of $\fmod kQ$. We need the following lemmas.

\begin{lem}[{\cite[Theorem 2.3]{sch}}]\label{lem:sch}
Let $E\in\fmod kQ$ be an exceptional module. There exists a finite acyclic quiver $Q'$ with $n-1$ vertices  such that 
$E_{\fmod}^{\perp}$
is equivalent to $\fmod kQ'$.
\end{lem}

\begin{lem}[{\cite[Lemma 5]{cb}}]\label{lem:cb}
Let $\cE=\{E_1,\hdots,E_r\}\subset \fmod kQ$. If $E_1,\hdots,E_r$ forms an exceptional sequence in $\Db(kQ)$, then there exists a finite acyclic quiver $Q'$ with $r$ vertices such that $\wide(\cE)\cong \fmod kQ'$. 
\end{lem}

{We say that a wide subcategory $\scrW\subsetneq\scrA$ of an abelian category $\scrA$ is {\it maximal} if there is no wide subcategory $\scrX$ with $\scrW\subsetneq\scrX\subsetneq\scrA$.}
We reformulate {\cite[Theorem 3.14]{asai}} -- it is the main ingredient to show the JD property of $\Db(kQ)$.

\begin{thm}[{\cite[Theorem 3.14]{asai}}]\label{thm:asai}
Let $\cS=\{B_1,\hdots,B_r\}$ be a semibrick in $\fmod kQ$ such that the wide subcategory {$\wide(\cS)\subsetneq \fmod kQ$} is maximal. Then each brick $B_i$ is exceptional for $1\leq i\leq r$.
\end{thm}

Using the above results, we prove the following.

\begin{prop}\label{prop:wild max}
Assume that $Q$ has $n>1$ vertices. Let $\scrM\subsetneq\Db(kQ)$ be a  maximal thick subcategory. If $\scrM$ is finitely generated, it is generated by an exceptional sequence $E_1,\hdots,E_{n-1}\in\scrM \cap \fmod kQ$ of length $n-1$. 
\begin{proof}
We use induction on $n$. Assume $n=2$. By Proposition \ref{prop:wide thick} and Proposition \ref{prop:unique semibrick}, $\scrM=\lbr\cS\rbr$ for a unique semibrick $\cS$ such that the wide subcategory $\wide(\cS) \subsetneq \fmod kQ$ is maximal. By Theorem \ref{thm:asai}, $\scrM$ contains an exceptional object $E\in \fmod kQ$. Then there is a semi-orthogonal decomposition 
\[
\Db(kQ)=\l E^{\perp},E\r,
\]
where $E^{\perp}=\lbr E_{\fmod}^{\perp}\rbr$.
By Lemma \ref{lem:sch}, there is an exact equivalence $E_{\fmod}^{\perp}\cong \fmod k$. Thus, the wide subcategory $E_{\fmod}^{\perp}\subset \fmod kQ$ is simple as it is generated by an exceptional object $F\in\fmod kQ$. Therefore, 
the thick subcategory $\lbr E\rbr\subset \Db(kQ)$ is  maximal. Since $\lbr E\rbr\subseteq \scrM$, we obtain $\lbr E\rbr=\scrM$, which proves the case when $n=2$.

Assume that the result holds for $n-1$.  By a similar argument as above,  we see that $\scrM$ contains an exceptional module $E\in \fmod kQ$, and there is a finite acyclic quiver $Q'$ with $n-1$ vertices such that $E_{\fmod}^{\perp}\cong \fmod kQ'$. Put
\[
\scrM'\defeq \scrM\cap \lbr E_{\fmod}^{\perp}\rbr\subseteq \lbr E_{\fmod}^{\perp}\rbr,
\]
and consider $\scrM'$ as a thick subcategory of $\Db(kQ')$ via a natural equivalence $\lbr E_{\fmod}^{\perp}\rbr\cong \Db(kQ')$. We define a functor $F\colon \Db(kQ)\to\Db(kQ')$ by the composition
\[
 \Db(kQ)\surjto \Db(kQ)/\lbr E\rbr\simto\lbr E_{\fmod}^{\perp}\rbr\simto \Db(kQ'),
\]
where the first functor is the Verdier quotient functor. Since $F^{-1}(\scrM')=\scrM$, $\scrM'$ is maximal in $\Db(kQ')$ by Proposition \ref{verdier}. By the induction hypothesis, there is a full exceptional sequence $F_1,\hdots,F_{n-2}\in \fmod kQ'$ of $\scrM'$. By the semi-orthogonal decomposition
\[
\scrM=\l \scrM\cap \lbr E_{\fmod}^{\perp}\rbr,E\r,
\]
the sequence of modules $F_1,\hdots,F_{n-2},E\in \fmod kQ$ forms a full exceptional sequence of $\scrM$.
\end{proof}
\end{prop}

{Now we are ready to prove the JD property of $\Db(kQ)$ by classifying all composition series in $\Db(kQ)$. We say that a composition series is {\it exceptional} if it is built from a full exceptional sequence, i.e., it is of the form as in \eqref{eq:fec}.}

\begin{thm}\label{thm:JD wild} 
Every composition series in $\Db(kQ)$ is exceptional. In particular, $\Db(kQ)$ satisfies the JD property, and $\ell(\Db(kQ))=n$.
\begin{proof}
Let $\scrS_*=\bigl(0=\scrS_0\subsetneq \scrS_1\subsetneq\cdots\subsetneq\scrS_r=\Db(kQ)\bigr)\in\CS(\Db(kQ))$. 

First, we show that, for each $1\leq i\leq r$, the thick subcategory $\scrS_i$  has a full exceptional sequence. We prove this by induction on $n$. The case when $n=1$ is obvious, since $\Db(kQ)\cong \Db(k)$. 
Since $\scrS_{r-1}\subsetneq \Db(kQ)$ is a finitely generated maximal thick subcategory,  Proposition \ref{prop:wild max} implies that there exists an exceptional sequence $\cE=\{E_1,\hdots, E_{n-1}\}\subset \fmod kQ$ such that
$\scrS_{r-1}=\lbr \cE\rbr$. By Lemma \ref{lem:cb}, there exists a finite acyclic quiver $Q'$ with $n-1$ vertices such that $\wide(\cE)\cong \fmod kQ'$. Using Proposition \ref{prop:wide thick}\,(2), we obtain a sequence of equivalences
\[
\scrS_{r-1}=\lbr \cE\rbr=\lbr \wide(\cE)\rbr\cong \Db(\wide(\cE))\cong \Db(kQ').
\]
If we put
$\scrS'_*\defeq\bigl(0\subsetneq \scrS_1\subsetneq\cdots\subsetneq\scrS_{r-1}\cong\Db(kQ')\bigr)$,
the ascending chain $\scrS_*'$ can be considered as a composition series in $\Db(kQ')$. By the induction hypothesis, for $1\leq i\leq r-2$, the thick subcategory $\scrS_i$ is generated by an exceptional sequence.

Since $\scrS_1$ is simple, there is an exceptional module $E_1$ such that $\scrS_1=\lbr E_1\rbr$. Then there is a semi-orthogonal decomposition 
\[
\scrS_2=\l E_1, ^{\perp}\!E_1\r.
\]
Since $\scrS_2$ is admissible in $\Db(kQ)$, so is  $^{\perp}\!E_1\subset \Db(kQ)$. Hence $^{\perp}\!E_1$  is generated by an exceptional sequence by \cite[Theorem A.4]{hk2} (cf.\ \cite[Corollary 3.7]{robotis}). Since $^{\perp}\!E_1\cong \scrS_2/\scrS_1$ is simple, there is an exceptional module $E_2\in \fmod kQ$ such that $E_1,E_2$ is a full  exceptional sequence of $\scrS_2$. Repeating this argument, we obtain the first assertion. The second assertion follows since the length of an exceptional composition series in $\Db(kQ)$ is equal to $\rk K_0(\Db(kQ))=n$.
\end{proof}
\end{thm}

\begin{rem}\label{rem:cps vs nc}
Theorem \ref{thm:JD wild} implies that every composite subcategory of $\Db(kQ)$ is generated by an exceptional sequence of modules. By \cite[Theorem 1.2]{hk2}, there is an isomorphism of posets
\[
\Th^{\rm cps}(\Db(kQ))\cong {\rm NC}(K_0(\fmod kQ)),
\]
where the left hand side is the poset of composite subcategories of $\Db(kQ)$ and the right hand side is the poset of {\it non-crossing partitions} of a generalized Cartan lattice defined on $K_0(\fmod kQ)$. See loc. cit. for the details.
\end{rem}

\subsection{Application of results on hereditary algebras}~

In this section, we provide several applications of Theorem \ref{thm:JD wild}.

\begin{cor}\label{cor:wild kronecker}
Let $\scrT$ be a  
triangulated category with a dg-enhancement. If $\scrT$ admits 
a full strong exceptional sequence  $E_1,E_2$ of length two, then $\scrT$ satisfies the JD property.
\begin{proof}
Since $T\defeq E_1\oplus E_2$ is a tilting generator of $\scrT$, there is an equivalence
\[
\scrT\cong \Perf \End(T)
\]
by Theorem \ref{thm:keller equiv}\,(2).  
If we set $d\defeq \dim \Hom(E_1,E_2)$, then $\End(T)$ is isomorphic to the path algebra $kK_d$ of a $d$-Kronecker quiver
\[
K_d\colon \hspace{3mm}\begin{tikzcd}
    1
    \arrow[r, draw=none, "\raisebox{+1.5ex}{\vdots}" description]
    \arrow[r, bend left,        "\alpha_1"]
    \arrow[r, bend right, swap, "\alpha_{d}"]
    &
    2.
\end{tikzcd}
\]
So $\scrT\cong \Perf(kK_d) \cong \Db(kK_d)$ satisfies the JD property by Theorem \ref{thm:JD wild}.
\end{proof}
\end{cor}

\begin{rem}
In Proposition \ref{prop:kronecker}, we will see that the assumption of the strongness of the exceptional sequence in Corollary \ref{cor:wild kronecker} is necessary. 
\end{rem}

We will see that the following shows that the derived categories of the projective plane and Hirzebruch surfaces have a different feature.

\begin{cor} \label{cor:admissibleJDP}
Let $0\neq\scrA\subsetneq \Db(\bP^2)$ be a non-trivial admissible subcategory. Then $\scrA$ satisfies the JD property. 
\begin{proof} If $\scrA$ is generated by an exceptional object, it is simple and thus satisfies the JD property.
Otherwise, by \cite{pir}, we can assume that $\scrA$ is generated by an exceptional sequence $E_1,E_2$ of length two, which is a subcollection of a mutation of the full strong exceptional sequence $\cO,\cO(1),\cO(2)$ of $\Db(\bP^2)$. By \cite[Proposition 3.3, Corollary 2.4]{bp2}, all mutations of $\cO,\cO(1),\cO(2)$ are again strong. Therefore, the exceptional sequence  $E_1,E_2$ is strong, and so $\scrA$ satisfies the JD property by Corollary \ref{cor:wild kronecker}.
\end{proof}
\end{cor}

The following can be seen as a (partial) extension of the corollary above. This might be an indication that the JD property holds for the surfaces in Krah's construction, cf.\ Section \ref{section:Krah's phantom}.

\begin{rem} \label{rem:SODintoJDP}
Let $n \in \bZ_{>0}$, and  
let $X$ be the blow-up of $\bP^2$ in $n$ general closed points. Then there is a semi-orthogonal decomposition
\begin{align*}
\Db(X)=\langle \Db(kT_n), \Db(kK_3) \rangle
\end{align*}
where $T_n$ is a connected quiver with $n$ vertices
and $K_3$ is the $3$-Kronecker quiver, cf.\ \cite[Proof of Proposition 4.2\,(2)]{BelRae} -- here, $\Db(kK_3)$ is generated by the pullbacks of $\cO(1)$ and $\cO(2)$.
In particular, in combination with Theorem \ref{thm:JD wild}, we see that $\Db(X)$ decomposes into two admissible subcategories that satisfy the JD property. One subcategory has a Grothendieck group of rank $n$ and the other has a Grothendieck group of rank $2$. 
Further examples of similar semi-orthogonal decompositions for more general surfaces have also been obtained in \cite{TU}, see in particular, Example 7.4 in loc. cit.
\end{rem}

\noindent
We use Theorem \ref{thm:JD wild} to show the JD property for stack quotients of $\bP^1$.

\begin{cor}\label{C:QuotStacks}
Let $G \subset \mathsf{PGL}_2(\bC)$ be a finite subgroup. Let $\mathbb{X}:=[\bP^1/G]$ be the quotient stack. Then $\Db(\mathbb{X})$ satisfies the JD property.
\end{cor}
\begin{proof}
$\Db(\mathbb{X})$ has a full strong exceptional sequence with a hereditary endomorphism algebra, cf.\ \cite[Proposition 2.4]{gl}, which is formulated in the language of `Geigle--Lenzing weighted projective lines' (see e.g.\ \cite[Section 1.4]{pol} for a translation between our setting and \cite{gl}).
\end{proof}

\begin{rem}
Actually, Corollary \ref{C:QuotStacks}
holds more generally for all `\emph{domestic} Geigle--Lenzing weighted projective lines', cf.\ \cite[Proposition 2.4]{gl}.
\end{rem}

The {\it root category} of a hereditary algebra $\Lambda$ is defined to be the orbit category 
\[
\scrR_\Lambda\defeq \Db(\Lambda)/[2],
\]
which is introduced by Happel \cite{happel}. Happel showed that for a Dynkin quiver $Q$ of type ADE, the isomorphism classes of indecomposable objects in $\scrR_{kQ}$ correspond bijectively to the  root system of the simple Lie algebra of the same ADE type. The root category $\scrR_\Lambda$ has a natural triangulated structure, since so does $\Db(\Lambda)/[m]$ for any $m\in\bZ$ by \cite{kel}. Combining \cite{koe} and Theorem \ref{thm:JD wild}, shows that $\scrR_{\Lambda}$ satisfies the JD property.

\begin{cor}\label{cor:orbit}
Let $Q$ be a finite acyclic quiver, and let $m\in\bZ_{>0}$ be a positive integer. Then the natural projection functor $\pi\colon \Db(kQ)\to \Db(kQ)/[m]$ induces an isomorphism 
\[
\Th(\Db(kQ))\simto \Th(\Db(kQ)/[m]) \hspace{7mm} \scrU\mapsto \pi(\scrU).
\]
In particular, $\Db(kQ)/[m]$ satisfies the JD property.
\begin{proof}
The first assertion is a special case of \cite[Theorem 4.4]{koe}. The second one  follows immediately from the first and Theorem \ref{thm:JD wild}.
\end{proof}
\end{cor}

At the end of this section, we ask the following question.

\begin{ques}
Let $\scrA$ be a hereditary abelian category over $k$ such that  $\Db(\scrA)$ is of finite length. Does $\Db(\scrA)$ satisfy the JD property?
\end{ques}

\vspace{3mm}
\subsection{Singular varieties and singularity categories}~

We discuss the length of derived categories of some singular varieties. First, we consider perfect complexes over artinian rings. 

\begin{prop}\label{prop:artin length}
Let $R$ be a commutative artinian ring, and 
denote by $n$  the number of maximal ideals in $R$. Then $\ell(\Perf R)=n$, and $\Perf R$ satisfies the JD property.
\begin{proof}
This follows from Theorem \ref{class ring} and Lemma \ref{lem:matsui}.
\end{proof}
\end{prop}

Let $X$ be a regular noetherian separated scheme,  $L$ an ample line bundle, and $s\in \Gamma(X,L)$ a non-zero-divisor. Denote by $Z$ the zero scheme of $s$, and  consider the singularity category
\[
\Dsg(Z)\defeq \Db(Z)/\Perf Z.
\]
For an object $F\in \Dsg(Z)$, we define its support by
\[
\Supp(F)\defeq\{p\in Z\mid F_{p}\not\cong 0 \mbox{ in } \Dsg(\cO_{Z,p})\}.
\]
It is known that $\Supp(F)$ is a closed subset of the singular locus $\Sing(Z)$. 
For a specialization closed subset $W\subseteq \Sing(Z)$, denote by 
$\D^{\rm sg}_W(Z)\subseteq \Dsg(Z)$ the thick subcategory consisting of objects $F$ with $\Supp(F)\subseteq W$. The following is a global version of Takahashi's result \cite{takahashi1}.

\begin{thm}[{\cite{hirano},\cite{stevenson}}]\label{thm:sing lattice}
The map 
\[
f\colon \{W\subseteq \Sing(Z)\mid \mbox{$W$ is specialization-closed}\}\to \Th(\Dsg(Z))
\]
given by $f(W)\defeq \D^{\rm sg}_W(Z)$ is an order-preserving bijection.
\end{thm}

\begin{rem}
If $L$ is not ample, we need to restrict $\Th(\Dsg(Z))$ to the set of thick subcategories closed under certain tensor action (see \cite{hirano,stevenson}).
\end{rem}

The combination of Theorem \ref{thm:sing lattice} and Lemma \ref{lem:matsui} implies the following.

\begin{cor}\label{prop:sing length}
We keep the notation from above. If $\ell(\Dsg(Z))$ or $\#\Sing(Z)$ is finite, then 
$\ell(\Dsg(Z))=\#\Sing(Z)$  and
$\Dsg(Z)$ has the JD property. 
\end{cor}

\begin{cor}\label{prop:artin derived}
Let $R$ be an artinian hypersurface singularity.  
Denote by $n$  the number of maximal ideals of $R$, and denote by $m$ the number of singular points. Then there is a composition series $\scrS_*\in \CS(\Db(R))$ with $\ell(\scrS_*)=n+m$. In particular, $\ell(\Db(R))\leq n+m$.
\begin{proof}
By Proposition \ref{prop:artin length}, $\ell(\Perf R)=n$ holds, and  by Corollary \ref{prop:sing length} we have $\ell(\Db(R)/\Perf R)=m$. Hence $\ell(\Db(R))\leq n+m$ by Proposition \ref{prop:length verdier}.
\end{proof}
\end{cor}

\begin{exa}\label{exa:singular} Using the above results, we discuss the length of derived categories of singular varieties.
\begin{itemize}
\item[(1)] Let $R\defeq k[x]/\l x^n\r$ for $n>1$. Then $R$ is artinian, and $\Spec R=\Sing R$ has a unique point. Thus $\ell(\Db(R))\leq2$ by Proposition \ref{prop:artin derived}. Since $\Db(R)$ is not simple, 
$\ell(\Db(R))=2$.

\item[(2)]Let $X$ be the nodal quadric in $\bP^3$ -- note that this projective surface can also be described as the weighted projective plane $\bP(1,1,2)$. 

Let $R\defeq k[x]/\l x^2\r$. There is a semi-orthogonal decomposition 
\[
\Db(X)=\l \Db(R), \Db(k),\Db(k)\r
\]
by \cite{kuz}, \cite[Example 5.7]{kaw} or \cite[Example 5.12]{kks} , which induces a semi-orthogonal decomposition 
\[
\Perf X=\l \Perf R, \Db(k), \Db(k)\r.
\] 
These decompositions induce several different composition series of length $4$ in $\Db(X)$ -- in particular, $\ell(\Db(X))\leq 4$.
\begin{align*}
&\Bigl(0\subsetneq\Perf R\subsetneq \Db(R) \subsetneq \l\Db(R),\Db(k)\r\subsetneq \Db(X)\Bigr)\\
&\Bigl(0\subsetneq \Perf R\subsetneq \l\Perf R,\Db(k)\r\subsetneq \Perf X \subsetneq\Db(X)\Bigr)\\
&\Bigl(0\subsetneq \Db(k)\subsetneq \l\Db(k),\Db(k)\r\subsetneq \Perf X \subsetneq\Db(X)\Bigr)
\end{align*}
In the last two cases, we use that $\Dsg(X)$ is simple (by Corollary \ref{prop:sing length}), since
$X$ has a unique isolated hypersurface singularity.

\item[(3)] Let $X\defeq \bP(1,2,3)$. By \cite[Example 5.8]{kaw} and \cite[Example 5.13]{kks}, there is a semi-orthogonal decomposition
\[
\Db(X)=\l \Db(k),\Db(k[x]/\l x^2\r), \Db(k[x]/\l x^3\r)\r,
\]
which induces a similar semi-orthogonal decomposition of $\Perf X$.
Hence $\ell(\Db(X))\leq 5$ and $\ell(\Perf X)\leq 3$.
\end{itemize}
\end{exa}

\begin{rem}\label{rem:rank 2}
By Corollary \ref{prop:sing length}, the singularity categories of local hypersurface rings with  isolated singularities are simple. Let $n$ be a positive integer. By \cite[Lemma 2.22]{kps}, there exists a local ring $R$ with an isolated singularity, such that  $\rk(K_0(\Dsg(R)))=n$. For example, if $R=k\lbr x,y,z,w\rbr/\l x^2+y^2+z^2w+w^{2r-1}\r$, then $\Dsg(R)$ is simple and $K_0(\Dsg(R))=\bZ^{\oplus 2}$.
\end{rem}

Let $(R,\m)$ be a commutative noetherian local ring. We write $\Spec_0R\defeq \Spec R\backslash\{\m\}$, and set
\[
\D_{\m}^{\rm sg}(R)\defeq \{M\in \Dsg(R)\mid M_{\p}\cong 0 \in \Dsg(R_{\p})\mbox{ for all $\p\in \Spec _0R$}\}.
\]
It is easy to see that $\D_{\m}^{\rm sg}(R)$ is a thick subcategory, and the following holds.
\begin{lem}\label{lem:dsgm} We have the following.
\begin{itemize} 
\item[$(1)$] The equality $\D_{\m}^{\rm sg}(R)=\lbr R/\m\rbr$ holds.
\item[$(2)$]$\D_{\m}^{\rm sg}(R)$ is non-zero if and only if $\Dsg(R)$ is non-zero.
\end{itemize}
\begin{proof}
(1) This is \cite[Corollary 4.3\,(3)]{takahashi3}.\\
(2) $(\Rightarrow)$ is obvious. If  $\Dsg(R)$ is non-zero, $\Sing(R)\neq\emptyset$. Since $R$ is singular, we have ${\rm proj.dim} \,R/\m={\rm gl.dim} \,R=\infty$. This shows that $R/\m$ is non-zero in $\Dsg(R)$. By (1), $\D_{\m}^{\rm sg}(R)$ is non-zero.
\end{proof}
\end{lem}
The following is well known, but we give a proof for the reader's convenience. 

\begin{lem}\label{lem:isolated}
Let $(R,\m)$ be a commutative noetherian local ring. Then $R$ is an isolated singularity if and only if $\Dsg(R)=\D_{\m}^{\rm sg}(R)$. 
\begin{proof}
($\Rightarrow$): By Lemma \ref{lem:dsgm}\,(1), $\D_{\m}^{\rm sg}(R)=\lbr R/\m\rbr$ holds. Thus if $R$ is an isolated singularity, $\Dsg(R)=\D_{\m}^{\rm sg}(R)$ holds by \cite[Corollary 4.3\,(2)]{takahashi3}.\\
($\Leftarrow$): If $\Sing R\neq \{\m\}$, there exists $\p\in \Spec_0 R$ such that $R_{\p}$ is singular. Then $R/\p\in \Dsg(R)$ does not lie in $\D_{\m}^{\rm sg}(R)$, since $(R/\p)_{\p}\not\cong 0$ in $\Dsg(R_{\p})$. This contradicts to the assumption $\Dsg(R)=\D_{\m}^{\rm sg}(R)$.
\end{proof}
\end{lem}

Following \cite{takahashi4}, we say that $R$ is {\it dominant} if, for each nonzero object $M\in \Dsg(R)$, the residue field $R/\m$ lies in the thick subcategory $\lbr M\rbr\subseteq \Dsg(R)$. The following is pointed out by Souvik Dey.

\begin{prop}
Let $R$ be a commutative noetherian local ring. Then the following are equivalent.
\begin{itemize}
\item[$(1)$] $\Dsg(R)$ is simple.
\item[$(2)$] $R$ is dominant isolated singularity.
\end{itemize}
\begin{proof}
(1)$\Rightarrow$(2): By assumption (1), $\Dsg(R)$ is simple. In particular, $\Dsg(R)$ is non-zero, $\D_{\m}^{\rm sg}(R)$ is a non-zero thick subcategory by Lemma \ref{lem:dsgm}\,(2). Since $\Dsg(R)$ is simple, $\Dsg(R)=\D_{\m}^{\rm sg}(R)$ holds. Therefore, $R$ is an isolated singularity by Lemma \ref{lem:isolated}.\\
(2)$\Rightarrow$(1): Let $M\in\Dsg(R)$ be a non-zero object. Since $R$ is dominant, $\D_{\m}^{\rm sg}(R)=\lbr R/\m\rbr\subseteq \lbr M\rbr\subseteq \Dsg(R)$. Since $R$ is an isolated singularity, this implies that $\lbr M\rbr=\Dsg(R)$ by Lemma \ref{lem:isolated}. This finishes the proof.
\end{proof}
\end{prop}

\noindent
The following is pointed out by Hiroki Matsui and Souvik Dey independently.
\begin{rem}
Let $R$ be a regular local ring, and let $R/I$ be a complete intersection of codimension $c$. By \cite[Corollary 10.5]{stevenson} and Lemma \ref{lem:matsui}, if $\ell(\Dsg(R/I))<\infty$, $R/I$ is a hypersurface singularity, i.e., $c=1$.
\end{rem}

At the end of this section, we discuss the JD property for certain $\bZ$-graded singularity categories.

\begin{cor} 
\begin{itemize}
\item[$(1)$]
Let $f\in S\defeq \bC[x_1,x_2,x_3]$ be a polynomial of type ADE. Define a $\bZ$-grading on $S$ by $d_i\defeq \deg(x_i)\in\bZ_{>0}$  such that $f\in S$ is homogeneous and $(d_1,d_2,d_3)$ is coprime. Then the graded singularity category $\D^{\rm sg}_{\bZ}(S/\l f\r)$ satisfies the JD property.

\item[$(2)$] Let $S\defeq k[x_1,x_2,x_3]$ with all $x_i$ in degree $1$ and let $C_3 = \langle \sigma  \rangle$ be the cyclic group of order three. Consider the action of $C_3$ on $S$ via $\sigma x_i = \omega x_i$ where $\omega^3 = 1$, $\omega \neq 1$. Let 
$R = S^{C_3}$ be the invariant ring. Then the graded singularity category $\D^{\rm sg}_{\bZ}(R)$
satisfies the JD property.
\end{itemize}
\begin{proof}
$(1)$ By \cite{orlov3}, $\D^{\rm sg}_{\bZ}(S/\l f\r)$ is equivalent to the homotopy category of graded matrix factorizations of $f$, and it is equivalent to the derived category $\Db(kQ)$ for a Dynkin quiver $Q$ of the same ADE type as $f$ by \cite{kst}. By  Theorem \ref{thm:JD wild}, it satisfies the JD property.

\noindent
$(2)$ Using e.g.\ \cite{KMV} there is an equivalence $\D^{\rm sg}_{\bZ}(R) \cong \Db(kK_3)$ where $K_3$ denotes the $3$-Kronecker quiver. Now the result follows from Theorem \ref{thm:JD wild}.
\end{proof}
\end{cor}

\vspace{3mm}
\section{Derived categories without the JD property}\label{section:JD property}

\noindent
Throughout this section, let $k$ be an algebraically closed field of characteristic zero. We give examples of derived categories of smooth varieties and finite-dimensional algebras of finite global dimension without the JD property. 

\vspace{2mm}
\subsection{Toric surfaces}\label{section:rational surfaces}
We show that derived categories of certain smooth projective toric surfaces, including Hirzebruch surfaces $\bF_d$ with $d>1$, don't satisfy the JD property. 
To this end,  we first consider the {\it graded $m$-Kronecker quiver} $\Kr^m_q$ of degree $q$ for $m>1$ and $q\in\bZ$, which is defined as the graded quiver
\[
\begin{tikzcd}
    1
    \arrow[r, draw=none, "\raisebox{+1.5ex}{\vdots}" description]
    \arrow[r, bend left,        "\alpha_0"]
    \arrow[r, bend right, swap, "\alpha_{m-1}"]
    &
    2
\end{tikzcd}
\]
where $\deg(\alpha_0)=0$ and $\deg(\alpha_i)=q$ for $i>0$. Then the graded path algebra $k\Kr_q^m$ is a graded $R$-algebra, where $R\defeq ke_1\times ke_2\subset k\Kr_q^m$ is the subalgebra generated by the idempotent elements.  

\begin{lem}\label{lem:kronecker formal}
The graded $R$-algebra $k\Kr^m_q$ is intrinsically formal.
\begin{proof}
Since the quiver of $k\Kr^m_q$ has only $2$ vertices, this follows as in the proof of \cite[Proposition A.3]{kuz3}.
\end{proof}
\end{lem}

Considering  $k\Kr^m_q$ as a dg-$R$-algebra with trivial differential, we define
\[
\scrKr^m_q\defeq \Perf(k\Kr^m_q),
\]
and call it the {\it graded $m$-Kronecker quiver category} of degree $q$ following \cite{ks}. If the underlying quiver is not wild, that is, if $m=2$, we simply write $\scrKr_q$.

\begin{lem}\label{lem:kronecker equiv}
Let $\scrT$ be an idempotent complete triangulated category with a dg-enhancement. Let $E_1,E_2\in\scrT$ be an exceptional sequence, and assume that $\Hom(E_1,E_2)\cong k$. Furthermore, assume  that there is  $q\neq0\in\bZ$ such that $\Hom(E_1,E_2[i])=0$ if and only if $i\not\in \{0,q\}$. Then there is an equivalence
\[
\l E_1,E_2\r\cong \scrKr^m_{q},
\]
 where $m\defeq \dim\Hom(E_1,E_2[q])+1$. 
\begin{proof} There is an equivalence $\l E_1,E_2\r\cong \Perf A$ by Theorem \ref{thm:keller equiv}\,(2), where 
 $A\defeq \REnd(E_1\oplus E_2)$ is the dg-endomorphism algebra of $E_1\oplus E_2$. This dg-algebra contains the subalgebra $R\defeq k\id_{E_1}\times k\id_{E_2}\cong k^2$, and the cohomology algebra $H^*(A)$ is isomorphic to the graded $R$-algebra $k\Kr^m_q$. By Lemma \ref{lem:kronecker formal}, $A$ is quasi-isomorphic to $k\Kr^m_q$. Consequently, $\l E_1,E_2\r\cong \Perf A\cong \scrKr^m_q$.
\end{proof}
\end{lem}

\begin{lem}\label{lem:-2 curve}
Let $S$ be a smooth projective surface, and $E\subsetneq S$ be a smooth rational curve with $E^2=-m$ for some $m>1$. If $\cO_S\in \Db(S)$ is exceptional,  $\cO_S,\cO_S(E)$ is an exceptional sequence in $\Db(S)$. Moreover the admissible subcategory $\scrC\defeq\l \cO_S,\cO_S(E)\r\subseteq \Db(S)$ is equivalent to $\scrKr^m_1$.
\begin{proof}
A standard computation shows the first assertion and the following equality
\[
\Ext^*(\cO_S,\cO_S(E))\cong k\oplus (k^{m-1}[-1]),
\]
which induces an equivalence $\scrC\cong \scrKr^{m}_1$ by Lemma \ref{lem:kronecker equiv}.
\end{proof}
\end{lem}

\begin{prop} \label{prop:kronecker} Let $d \in \bZ_{>1}$. Then
 $\{2,3\}\subseteq\LS(\scrKr^d_1).$
\begin{proof}
By Lemma \ref{lem:kronecker equiv}, $\scrKr^d_1$ has a full exceptional sequence of length $2$, so $ 2 \in \LS(\scrKr^d_1)$.

Let $R_d=k[x_1,\hdots,x_{d-1}]/\l x_1,\hdots,x_{d-1}\r^2$.
It is well-known that  $\scrKr^d_1$ has a tilting object, cf.\ e.g.\ \cite[Proposition 1.5]{HP}.
By \cite[Theorem 6.26]{KK}, its endomorphism algebra is isomorphic to 
$B_d:=\End_{R_d}(R_d \oplus S_d)$, where $S_d$ is the simple $R_d$-module. In summary, there is a triangle equivalence
\begin{align}\label{eq:Equiv}
\scrKr^d_1 \cong \Db(B_d).
\end{align}

Let $e \in B_d$ be the primitive idempotent corresponding to the projection onto $R_d$, and let $S$ be the simple $B_d$-module corresponding to the idempotent $1-e$. Then there is a triangle equivalence, cf. e.g. \cite[Remark 2.9]{Kalckgentle}
\begin{align}\label{eq:VQ}
\Db(B_d)/\lbr S\rbr \cong \Db(R_d).
\end{align}
The latter category has length $2$ by \cite[Theorem A\,(2)]{el2} and Theorem \ref{class ring}.

Let $P_S$ be the projective cover of $S$. Then there is an isomorphism of $B_d$-modues $B_d \cong P_S \oplus P_e$, where $P_e$ is the indecomposable projective module corresponding to $e$.  The projective resolution 
\begin{align}
0 \to P_S^{\oplus d-1}  \to P_e \to P_S \to S \to 0
\end{align}
(cf.\ e.g.\ \cite[Theorem 2.8]{KK}) shows that $S$ is a $(d-1)$-bouquet $2$-spherelike object, so $\lbr S\rbr$ is simple, by Proposition \ref{prop:bouquet}. Together with \eqref{eq:Equiv} and \eqref{eq:VQ}, this shows $ 3 \in \LS(\scrKr^d_1)$ as claimed.
\end{proof}
\end{prop}

\begin{rem} Let $\Lambda(1, 2, 0)$ be the algebra introduced in Section \ref{subs: derived-discrete}.
Using the triangle equivalence
$
\scrKr^2_1 \cong \Db(\Lambda(1, 2, 0)),
$
a more detailed analysis shows 
\[\{2,3\} = \LS(\scrKr^2_1),\]
cf.\ also \cite{bro}.

\end{rem}

 Let $X$ be a smooth projective toric surface associated to a fan $\Sigma$. Let $\tau_1,\hdots,\tau_n$ be the set of rays in $\Sigma$, where the order of $\tau_i$ is counterclockwise, and denote by $E_i$ the irreducibel torus-invariant divisor corresponding to $\tau_i$. Then each $E_i$ is a smooth rational curve. {We say that a ray $\tau_i$ is {\it adjacent} to $\tau_j$ if $|i-j|\in \{1,n-1\}$.} 

\begin{thm}\label{thm:toric}
We keep the notation from above. 
 Assume that there is a sequence $E_{i_1},\hdots,E_{i_r}$ satisfying the following conditions:
\begin{itemize}
\item[$(1)$] For each $1\leq j\leq r$, $m_j\defeq -E_{i_j}^2>1$.
\item[$(2)$]There are no adjacent rays in $\tau_{i_1},\hdots,\tau_{i_r}$. 
\end{itemize}
Then there are admissible subcategories $\scrC_1,\hdots,\scrC_r$ and $\scrL_1,\hdots,\scrL_{r+1}$ such that $\scrC_j\cong \scrKr^{m_j}_1$, $\scrL_i$ is generated by an exceptional sequence of line bundles if $\scrL_i\neq0$, and there is a semi-orthogonal decomposition
\[
\Db(X)=\l \scrL_1,\scrC_1,\scrL_2,\scrC_2\hdots,\scrL_{r},\scrC_r,\scrL_{r+1} \r.
\]
In particular, we obtain
\[\{n,n+1,\hdots,n+r\}\subseteq \LS(\Db(X)),\]
and  $\JD(\Db(X))\geq r$ holds.
\begin{proof}
By \cite[Theorem 5.1]{hil}, $\Db(S)$ admits a full exceptional sequence
\[L_1,\hdots L_n\]
of line bundles, where $L_1\defeq \cO_S$ and  $L_i\defeq \cO\bigl(\sum_{j=1}^{i-1}E_j\bigr)$ for $i\geq2$. If we put   $\scrC_{j}\defeq\l L_{i_j},L_{i_j+1}\r$ and $m_j\defeq -E_{i_j}^2$,  there is a semi-orthogonal decomposition
\[
\hspace{10mm}\Db(S)=\l L_1,\hdots,L_{i_1-1},\scrC_{1},L_{i_1+2},\hdots, L_{i_r-1},\scrC_{r},L_{i_r+2},\hdots,L_{n}\r.
\]
By standard computation, $\Ext^*(L_{i_j},L_{i_j+1})\cong k\oplus (k^{m_j-1}[-1])$ holds. This implies an equivalence $\scrC_{j}\cong \scrKr^{m_j}_1$ by Lemma \ref{lem:-2 curve}.
 Combining Proposition \ref{prop:kronecker} with the semi-orthogonal decomposition shows that 
$\Db(S)$ admits composition series of lengths $n,n+1,\hdots,n+r$.
\end{proof}
\end{thm}

Recall from Definition \ref{dfn:composite} that a thick subcategory $\scrU\subseteq \Db(X)$ is composite in $\Db(X)$ if and only if $\ell(\scrU)<\infty$ and $\ell(\Db(X)/\scrU)<\infty$. If a composite thick subcategory of $\Db(X)$ fails the JD property, so does $\Db(X)$.

\begin{cor}\label{cor:toric surface}
If a smooth projective toric surface $X$ has a smooth rational curve $E$ with $E^2<-1$, then $\Db(X)$ contains a composite admissible subcategory without the JD property. In particular, $\Db(X)$
does not satisfy the JD property.
\begin{proof}
This follows from Theorem \ref{thm:toric} and the following lemma.
\end{proof}
\end{cor}

\begin{lem}
Notation is the same as above. Every irreducible curve $D$ with $D^2<0$ on $X$ is torus-invariant.
\begin{proof}
By \cite[Theorem 4.1.3]{cls}, there are torus invariant irreducible curves $D_1,\hdots,D_m$ such that $D$ is linearly equivalent to $\sum_{i=1}^ma_iD_i$ with $a_i\geq0$. Since $D^2=\sum_{i=1}^ma_i(D. D_i)<0$, there is $i$ with $D.D_i<0$. This implies $D=D_i$, since two different irreducible curves have a nonnegative intersection number \cite[Proposition V.1.4]{har}.
\end{proof}
\end{lem}

Corollary \ref{cor:toric surface} can be generalised to the following situation. 
\begin{prop}
Let $X$ be a smooth projective surface with $\cO_X\in\Db(X)$ exceptional, and let $E$ be a smooth rational curve with $E^2<-1$. If $\l\cO_X,\cO_X(E)\r$ is composite, $\Db(X)$ does not satisfy the JD property.
\begin{proof}
By Lemma \ref{lem:-2 curve}, there is an equivalence $\l\cO_X,\cO_X(E)\r\cong \scrKr^m_1$, where $m\defeq -E^2$. By Proposition \ref{prop:kronecker}, the admissible subcategory $\l\cO_X,\cO_X(E)\r$ does not satisfy the JD property. Since $\l\cO_X,\cO_X(E)\r$ is composite, $\Db(X)$ does not satisfy the JD property.
\end{proof}
\end{prop}

\vspace{2mm}
\subsection{Derived-discrete algebras}~\label{subs: derived-discrete}

Let $\Omega=\{(r,n,m)\in\mathbb{Z}^3|n\geq r\geq 1,m\geq 0\}$  and let
\[\Lambda(r,n,m)\defeq kQ(r,n,m)/I(r,n,m)\] for $(r,n,m)\in\Omega$, where $Q(r,n,m)$ is the quiver
\[
\xymatrix@R=1.2pc@C=1.2pc{&&&&& \scriptstyle{1}\ar[r]^{\scriptstyle{\alpha_1}}&\cdot\ar@{.}[r]&\cdot\ar[r]^(.35){\scriptstyle{\alpha_{n-r-2}}} & \scriptstyle{n-r-1}\ar[rd]^{\scriptstyle{\alpha_{n-r-1}}}&\\
\scriptstyle{(-m)}\ar[r]^(.6){\scriptstyle{\alpha_{-m}}}&\cdot\ar@{.}[r]&\cdot\ar[r]^(.4){\scriptstyle{\alpha_{-2}}}&\scriptstyle{(-1)}\ar[r]^(.6){\scriptstyle{\alpha_{-1}}} & \scriptstyle{0}\ar[ru]^{\scriptstyle{\alpha_0}} &&&& & \scriptstyle{n-r}\ar[ld]^{\scriptstyle{\alpha_{n-r}}}\\
&&&&&\scriptstyle{n-1}\ar[lu]^{\scriptstyle{\alpha_{n-1}}} &
\cdot\ar[l]^(.35){\scriptstyle{\alpha_{n-2}}}&\cdot\ar@{.}[l]
&\scriptstyle{n-r+1}\ar[l]^(.65){\scriptstyle{\alpha_{n-r+1}}}}
\]
and $I(r,n,m)$ is the two-sided ideal of $kQ(r,n,m)$ generated by the paths
$\alpha_0\alpha_{n-1}$, $\alpha_{n-1}\alpha_{n-2}$, $\ldots$,
$\alpha_{n-r+1}\alpha_{n-r}$. Then the algebra $\Lambda(r,n,m)$ is  finite-dimensional, and it is of finite global dimension if and only if $r<n$. 

Let $\Lambda$ be a finite-dimensional $k$-algebra. We say that $\Lambda$ is {\it derived-discrete} if for every map $v\colon \bZ\to K_0(\Db(\Lambda))$ there are only finitely many isomorphism classes of objects $F\in \Db(\Lambda)$ such that $[H^i(F)]=v(i)\in K_0(\Db(\Lambda))$ for all $i\in\bZ$. 
Recall that $\Lambda$ is {\it connected} if it is not a direct product of two algebras, and we say that $\Db(\Lambda)$ is of {\it Dynkin type} if it is equivalent to the derived category $\Db(kQ)$ of a Dynkin quiver $Q$. It is easy to see that if the derived category $\Db(\Lambda)$ is of Dynkin type, then $\Lambda$ is derived-discrete. The following is the classification of derived-discrete algebras by \cite{bgs}.

\begin{thm}[{\cite[Theorem A]{bgs}}]\label{thm:dd alg}
Let $\Lambda$ be a connected finite-dimensional $k$-algebra, and assume that $\Db(\Lambda)$ is not of Dynkin type. Then $\Lambda$ is derived-discrete if and only if $\Db(\Lambda)\cong \Db(\Lambda(r,n,m))$ for some $(r,n,m)\in \Omega$. Moreover, $\Db(\Lambda(r,n,m))\cong \Db(\Lambda(r',n',m'))$ if and only if $(r,n,m)=(r',n',m')$.
\end{thm}

The following lemma should also follow from a detailed analysis of \cite{bro}. For the convenience of the reader, we include a direct argument.

\begin{thm}\label{thm:appendix 2}
Let $(r,n,m) \in \Omega$ with $r<n$. The category $\Db(\Lambda(r,n,m))$  admits composition series of length $n+m$ and of length $n+m+1$.
\end{thm}
\begin{proof}
Write $\Lambda\defeq\Lambda(r,n,m)$.
It is well-known that the category $\Db(\Lambda)$ admits a full exceptional sequence, e.g.\ \cite[Proposition 7.6]{bpp}. This yields a composition series of length $n+m$ (since $Q(r,n,m)$ has $n+m$ vertices).

We now explain how to obtain a composition series of length $n+m+1$. The following claim is well-known to experts. We include the proof for convenience of the reader. Denote by $S_i$ the simple module corresponding to the vertex $i$. 

{\it Claim}: The simple $\Lambda$-modules $S_{-m}, S_{-m+1}, \ldots, S_0, S_1, \ldots, S_{n-r-1}$ form an exceptional sequence in $\Db(\Lambda)$. 

Indeed, we have projective resolutions
\begin{align}
0 \to P_{i+1} \to P_i \to S_i \to 0,
\end{align}
where $P_j$ is the indecomposable projective at vertex $j$. This shows that $\Ext^{>1}_{\Lambda}(S_i, S_j)=0$ for all $-m \leq i, j, \leq n-r-1$. Moreover, it implies for all $-m \leq i, j, \leq n-r-1$ that $\Ext^1_{\Lambda}(S_i, S_j)\neq 0$ iff $j=i+1$. 
Finally, since the $S_j$ are simple and $k$ is algebraically closed, we have $\dim_k \Hom_\Lambda(S_i, S_j)=\delta_{ij}$. Since the $S_i$ are modules we don't have negative Exts. Summing up, we have checked  the claim.

For any idempotent $e \in \Lambda$, the functor $\Hom_\Lambda(\Lambda e, -)\colon \fmod\Lambda \to \fmod e\Lambda e$ induces a short exact sequence of triangulated categories
\begin{align} \label{E:SESfromIdemp}
\lbr\fmod\Lambda/\l e\r\rbr \to \Db(\Lambda) \to \Db(e\Lambda e).
\end{align} 
Specifying $e$ to be the idempotent $1-(\sum_{i=-m}^{-1} e_i + \sum_{i=1}^{n-r-1} e_i) \in \Lambda$, where the $e_i$ are the primitive idempotents corresponding to the vertices, we obtain 
$\lbr\fmod \Lambda/\l e\r\rbr=\lbr S_{-m}, S_{-m+1}, \ldots, S_{-1}, S_1, \ldots, S_{n-r-1}\rbr$ and 
$e \Lambda e \cong \Lambda(r, r+1,0)$. Since $S_{-m}, S_{-m+1}, \ldots, S_{-1}, S_1, \ldots, S_{n-r-1}$ is an exceptional sequence, $\lbr\fmod\Lambda/\l e\r\rbr$ has a composition series of length $m+n-r-1$. Thus it suffices to show that $D^b(\Lambda(r, r+1,0))$ has a composition series of length $r+2$. 

Let us consider the sequence \eqref{E:SESfromIdemp} for $\Lambda=\Lambda(r, r+1,0)$ and $e=1-e_1$. By Proposition  \ref{prop:bouquet}, $\lbr\fmod\Lambda/\l e\r\rbr=\lbr S_1\rbr$ is simple. Indeed, applying $\Hom(-,S_1)$ to a projective resolution
\begin{align}
0 \to P_1 \to P_0 \to P_r \to \cdots \to P_{2} \to P_1 \to S_1\to 0
\end{align}
shows that  $S_1$ is $(r+1)$-sphere-like. 
 Now $e \Lambda e \cong \Lambda(r, r,0)$ has infinite global dimension. We claim that $\Perf(\Lambda(r, r,0))$ has a composition series of length $r$ and $\Dsg(\Lambda(r, r,0))\defeq\Db(\Lambda(r, r,0))/\Perf(\Lambda(r, r,0))$ is simple. Summing up, this shows that $\Db(\Lambda(r, r+1,0))$ has a composition series of length $1+r+1=r+2$ as claimed. 

We show the claim. We can assume that $r>1$ since we have already seen that $\Lambda(1, 1,0)\cong k[x]/\l x^2\r$ has length 2 by Example \ref{exa:singular}\,(1). 
Consider a complex
\[Q=(P_0 \to P_{r-1} \to P_{r-2} \to \cdots \to P_1)\in \Perf(\Lambda(r,r,0)),\]
which is $(1-r)$-sphere-like. Then $Q$ is right-orthogonal to the thick subcategory of $\Perf(\Lambda(r,r,0))$ generated by an exceptional sequence $P_0,P_{r-1},\hdots,P_2$. Indeed, all indecomposable projectives have length
2, and each simple appears precisely once as the head and precisely once as the socle of any projective since $\Lambda(r, r,0)$ is self-injective.
Therefore every morphism between indecomposable projectives  is either a scalar multiple of the identity or a scalar multiple of the multiplication with an arrow $\alpha_i$ ($0\leq i\leq r-1$). This shows that every morphism (in the category of cochain complexes) from $P_i$ to $Q$  is null-homotopic for each $i=0,r-1,\hdots,2$, and so $Q \in \l P_0,P_{r-1},\hdots,P_2 \r^\perp$. 
Since $P_1\in \lbr Q, P_0,P_{r-1},\hdots,P_2\rbr$, 
we have a semiorthogonal decomposition
\begin{align}
\Perf(\Lambda(r, r,0))&=\langle Q, P_0, P_{r-1}, P_{r-2}, \ldots, P_2 \rangle.
\end{align}
Thus this semiorthogonal decomposition yields a composition series of length $r$ as claimed, since $\lbr Q\rbr$ is simple by Proposition \ref{prop:bouquet}. 
Finally,  $\Dsg(\Lambda(r, r,0))$ is equivalent to the triangulated orbit category $\Db(k)/[r]$ cf.\ \cite{kal}, which is simple by Corollary \ref{cor:orbit}. 
\end{proof}

\begin{cor}\label{cor:appendix}
Let $\Lambda$ be a connected finite-dimensional $k$-algebra of finite global dimension, and assume that $\Lambda$ is derived-discrete. Then the following are equivalent:
\begin{itemize}
\item[$(1)$] $\Db(\Lambda)$ satisfies the JD property.
\item[$(2)$] $\Db(\Lambda)$ is of Dynkin type.
\end{itemize}
\begin{proof}
(2) $\Rightarrow$ (1) follows from Theorem \ref{thm:JD wild}. Assume that $\Db(\Lambda)$ is not of Dynkin type. Then by Theorem \ref{thm:dd alg}, $\Lambda$ is derived equivalent to $\Lambda(r,n,m)$ for some $(r,n,m)\in \Omega$. Since $\Lambda$ is of finite global dimension, so is $\Lambda(r,n,m)$. Hence $r<n$ holds, and $\Db(\Lambda)$ does not satisfy the JD property by Theorem \ref{thm:appendix 2}. This shows (1) $\Rightarrow$ (2). 
\end{proof}
\end{cor}

The following was independently also pointed out to us by Greg Stevenson.

\begin{rem}
By Theorem \cite[Theorem 3.3]{orl}, the above derived categories $\Db(\Lambda)$ of derived-discrete algebras can be realized as admissible subcategories $\scrA$ of derived categories of some smooth projective schemes. If $\ell(\scrA^{\perp})<\infty$
 and $\scrA\cong\Db(\Lambda)$ does not satisfy the JD property, the derived category of such a scheme does not satisfy the JD property. 
\end{rem}

\vspace{2mm}
\subsection{Threefolds}~

Let $Y$ be a nodal projective threefold, i.e.~all its singular points $y_1,\hdots,y_r\in Y$ are ordinary double points.  Assume that there is a (crepant) small resolution $X\to Y$  with exceptional curves $C_1,\hdots,C_r\subset X$. If we write
\[
K_i\defeq \cO_{C_i}(-1),
\]
it is a $3$-spherical object in $\Db(X)$, and the spherical objects $K_1,\hdots,K_r$ are orthogonal to each other. Denote by 
\[
\sfT_i\defeq\sfT_{K_i}\colon \Db(X)\simto\Db(X)
\]
the spherical twist associated to $K_i$. The goal of this section is to prove the following.

\begin{thm}\label{thm:3-fold}
Let $E_1,\hdots,E_r\in \Db(X)$ be an exceptional sequence  such that 
\begin{equation}\label{eqn:3-fold}
E_i|_{C_j}\cong \cO_{C_j}(\pm\delta_{ij}),
\end{equation}
 and put
$
\scrP\defeq\l \scrP_1,\hdots,\scrP_r\r,
$
where $\scrP_i\defeq\l E_i,\sfT_i(E_i)\r$. Assume that there is a composition series of length $\ell$ in $\Db(X)/\scrP$. Then 
\[
\{2r+\ell, 2r+1+\ell, \hdots, 3r+\ell\}\subseteq\LS(\Db(X)).
\]
In particular, $\Db(X)$ does not satisfy the JD property.
\end{thm}

The main ingredients of the proof are results from \cite{ks}.

\begin{prop}[\cite{ks}]\label{prop:3-fold}
Let $E_1,\hdots,E_r\in \Db(X)$ be an exceptional sequence that satisfies
\eqref{eqn:3-fold}.
Then the following holds.
\begin{itemize}
\item[$(1)$] For each $1\leq i\leq r$, the objects  $E_i,\sfT_i(E_i)$ form an exceptional sequence.
\item[$(2)$] The subcategory $\scrP_i\defeq\l E_i,\sfT_i(E_i)\r$ is equivalent to $\scrKr_2$.
\item[$(3)$] The sequence of subcategories $\scrP_1,\hdots,\scrP_r$ is semi-orthogonal in $\Db(X)$.
\end{itemize} 
\begin{proof}
By assumption \eqref{eqn:3-fold}, standard computation shows that 
\[
\dim\Ext^*(E_i,K_i)=\delta_{ij}.
\] Therefore, (1) and (2) follow from \cite[Lemma 3.10]{ks}, and (3) follows from \cite[Theorem 4.2\,(i)]{ks}.
\end{proof}
\end{prop}

For the proof of Theorem \ref{thm:3-fold}, we need to consider {\it categorical ordinary double points}, which we recall below. For $p\geq 0$, consider the following $\bZ$-graded ring 
\[
\sfA_p\defeq k[x]/\l x^2\r,
\]
where $\deg(x)=-p$. Considering $\sfA_p$ as a dg-algebra, we consider its derived category $\D(\sfA_p)$ and denote by $\D_{\rm fd}(\sfA_p)\subset \D(\sfA_p)$ the subcategory of dg-modules whose total cohomology is finite-dimensional. Then the following partial generalization of Example \ref{exa:singular}\,(1) holds.

\begin{lem}\label{lem:cat odp}
We have $\ell(\D_{\rm fd}(\sfA_p))=2$.
\begin{proof}
The perfect derived category $\Perf(\sfA_p)\subset \D_{\rm fd}(\sfA_p)$ is split generated by the free module $\sfA_p$, which is a $(-p)$-sphere-like object by construction of $\sfA_p$. By Proposition \ref{prop:bouquet}, $\Perf(\sfA_p)$ is simple. Moreover,  it is well known that $\Dsg(\sfA_p)\defeq\D_{\rm fd}(\sfA_p)/\Perf(\sfA_p)$ is also simple.  Indeed,  it is equivalent to the triangulated category $\vect^{L_p}(k)$ of $L_p$-graded finite-dimensional vector spaces, where $L_p\defeq\bZ/(p+1)\bZ$, the shift functor $[1]$ is the grading shift $(1)$ and triangles are (split) short exact sequences. The category $\vect^{L_p}(k)$ is split generated by a graded vector space of the form $k(i)$ for arbitrary $i\in L_p$, and every object in $\vect^{L_p}(k)$ contains $k(i)$ as a direct summand for some $i\in L_p$. Hence $\vect^{L_p}(k)$ is simple, and so is $\Dsg(\sfA_p)$.
\end{proof}
\end{lem}

The following is a partial generalization of Proposition \ref{prop:kronecker}. 

\begin{prop}\label{prop:grkronecker}
For $q\geq1$, the graded Kronecker quiver category $\scrKr_q$ has composition series of length $2$ and $3$. 
\begin{proof}
The category $\scrKr_q$ is generated by an exceptional sequence $E,E'$ consisting of the direct summands of the free module $k\Kr_q$ corresponding to each vertex of $\Kr_q$. Thus it has a composition series of length $2$. Define an object $K_+\in\scrKr_p$ by the following triangle
\[
K_+\to E\xrightarrow{\alpha_0}E'.
\]
By \cite[Lemma 3.5]{ks}, $K_+$ is a $(1+q)$-spherical object, and thus $\lbr K_+\rbr$ is simple. Combining Lemma \ref{lem:cat odp} with an equivalence $\scrKr_q/\lbr K_+\rbr\cong \D_{\rm fd}(\sfA_p)$ by \cite[Lemma 3.7]{ks}, we see that $\scrKr_q$ has a composition series of length $3$.
\end{proof}
\end{prop}

\begin{proof}[Proof of Theorem \ref{thm:3-fold}]
By Proposition \ref{prop:3-fold},  
\[
\scrP_1,\hdots,\scrP_r\subseteq \Db(X)
\]
is a semi-orthogonal sequence of admissible subcategories. Since each $\scrP_i$ is equivalent to $\scrKr_2$ by  Proposition \ref{prop:3-fold}\,(2), there are composition series of length $2r+\ell, 2r+1+\ell, \hdots, 3r+\ell$ by Proposition \ref{prop:grkronecker}.
\end{proof}

\begin{exa}
Let $Y_5$ be a smooth {\it quintic del Pezzo threefold}, i.e. a smooth projective variety such that $H\defeq-\frac{1}{2}K_{Y_5}$ is an ample generator of the Picard group and $H^3=5$.
 Consider the blow-up
\[
\sigma\colon\widetilde{Y_5}\to Y_5
\]
of $Y_5$ along a smooth rational curve  of degree $4$, and denote by $E\subset \widetilde{Y_5}$ the exceptional divisor of $\sigma$.  By \cite[Proposition 2.5]{ks2}, the base locus
\begin{equation}\label{eqn:curve}
C\defeq {\rm Bs}(|H-E|)\subset \widetilde{Y_5}
\end{equation}
of the linear system $|H-E|$ is a smooth rational curve, and the equalities  $H.C=1$ and $E.C=2$ hold.  Moreover, by \cite[Proposition 2.6\,(iii)]{ks2}, there is a  small resolution 
\[
\pi\colon \widetilde{Y_5}\to X
\]
of a Fano threefold $X$ with exactly one node, and the exceptional locus of $\pi$ is $C$ constructed  in \eqref{eqn:curve}. Then $D\defeq \cO_{\widetilde{Y_5}}(E-H)$ is an exceptional object and $D|_{C}\cong \cO_C(1)$ holds. Thus  $D, \sfT_{\cO_C(-1)}(D)$ is an exceptional sequence by Proposition \ref{prop:3-fold}.  Put 
\[
\scrP\defeq \l D, \sfT_{\cO_C(-1)}(D)\r.
\]
Then we claim that $\Db(\widetilde{Y_5})/\scrP\cong \,^{\perp}\scrP$ is of finite length. Indeed, the semi-orthogonal decomposition \cite[Equation (33)]{ks2} shows that there are an exceptional sequence $E_1, E_2\in \,^{\perp}\scrP$ and an equivalence
\[
\l E_1, E_2\r^{\perp}\cong \sigma^*(\scrB_{Y_5}),
\]
where the orthogonal on the left hand side is taken in $\,^{\perp}\scrP$, and $\scrB_{Y_5}\defeq \l \cO_{Y_5},\cO_{Y_5}(H)\r^{\perp}\subset \Db(Y_5)$. Since $\sigma^*$ is fully faithful, there is an equivalence $\scrB_{Y_5}\cong \sigma^*(\scrB_{Y_5})$. Furthermore, $\scrB_{Y_5}$ is equivalent to the derived category of finite-dimensional representations of the $3$-Kronecker quiver, and so $\scrB_{Y_5}$ has a full exceptional sequence of length two (see \cite[Section 1.1]{ks2}). Consequently, there is a composition series $\scrS_*\in \CS(\Db(\widetilde{Y_5})/\scrP)$ with $\ell(\scrS_*)=4$, and by Theorem \ref{thm:3-fold} 
\[
\{6,7\}\subseteq \LS(\Db(\widetilde{Y_5})).
\]
In particular, $\Db(\widetilde{Y_5})$ does not satisfy the JD property.
\end{exa}

\begin{rem}
The variety $\widetilde{Y_5}$ is not Fano, since $(-K_{\widetilde{Y}_5}).C=(2H-E).C=0$.
\end{rem}

\noindent
All known examples of smooth projective varieties whose derived categories don't have the JD property are not Fano. This leads to the following question.

\begin{ques}
Let $X$ be a smooth Fano variety. If $\ell(\Db(X))<\infty$, does $\Db(X)$ satisfy the JD property?
\end{ques}

\end{document}